\documentclass[11pt,
psamsfonts]{amsart}

\usepackage{amssymb,amsfonts,amsmath}
\usepackage[all,arc]{xy}
\usepackage{fullpage}
\usepackage{xspace}
\usepackage{amssymb, mathrsfs}
\usepackage{amsmath, graphicx, amsthm}
\usepackage{times}
\usepackage{pgf}
\usepackage{tikz}
\usepackage{tikz-cd}
\usetikzlibrary{matrix,arrows,decorations.pathmorphing}
\usepackage{enumerate}
\usepackage{hyperref}
\usepackage{cleveref}
\usepackage{mathtools}

\newtheorem{thm}{Theorem}[section]
\newtheorem{cor}[thm]{Corollary}
\newtheorem{pr}[thm]{Proposition}
\newtheorem{lm}[thm]{Lemma}

\theoremstyle{definition}
\newtheorem{defn}[thm]{Definition}

\newtheorem{ex}[thm]{Example}
\newtheorem{notn}[thm]{Notation}

\theoremstyle{remark}
\newtheorem{rem}[thm]{Remark}

\newcommand{\dis}{\displaystyle}
\newcommand{\D}{\mathcal{D}}

\newcommand{\R}{\mathbb{R}}
\newcommand{\F}{\mathbb{F}}
\newcommand{\ap} {\alpha}
\newcommand{\T}{\Theta}
\newcommand{\s} {\sigma}
\newcommand{\bt} {\beta}
\newcommand{\g} {\gamma}
\newcommand{\m}{\mathfrak{m}_}
\newcommand{\Om}{\Omega}
\newcommand{\om}{\omega}
\newcommand{\OL}{\Omega^1_{B|A}}
\newcommand{\ol}{\omega^1_{B|A}}
\newcommand{\ok} {\omega^1_{A}}
\newcommand{\OK} {\Omega^1_{A}}
\newcommand{\js} {\mathcal{J}_{\s}}
\newcommand{\is} {\mathcal{I}_{\s}}
\newcommand{\ns} {\mathcal{N}_{\s}}

\newcommand{\dn} {\overline{\Delta_N}}

\newcommand{\Z}{\mathbb{Z}}

\newcommand{\I}{\mathbb{I}}
\newcommand{\Lr}{\Leftrightarrow}

\newcommand{\Ra}{\Rightarrow}
\newcommand{\lb}{\left(}
\newcommand{\rb}{\right)}
 
\DeclareMathOperator{\Tr}{Tr}

\DeclareMathOperator{\ch}{char}
\DeclareMathOperator{\dl}{dlog} 
\DeclareMathOperator{\rsw}{rsw} 
\DeclareMathOperator{\Gal}{Gal}
\DeclareMathOperator{\Sw}{Sw}
\DeclareMathOperator{\Art}{Art}

\makeatletter
\let\c@equation\c@thm
\makeatother
\numberwithin{equation}{section}

\newcommand{\textlatin }

\title{RAMIFICATION THEORY FOR ARTIN-SCHREIER EXTENSIONS OF
VALUATION RINGS }

\author{VAIDEHEE THATTE}

\begin{document}

\maketitle

\begin{small}\begin{center}\textbf{Abstract}\end{center} 

The goal of this paper is to generalize and refine the classical
ramification theory of complete discrete valuation rings to more
general valuation rings, in the case of Artin-Schreier
extensions. We define refined versions of invariants of
ramification in the classical ramification theory and compare
them. Furthermore, we can treat the defect case.
\end{small} 
\section{Introduction}

We present a generalization and refinement of the classical
ramification theory of complete discrete valuation rings to
valuation rings satisfying either (I) or (II) (as explained in
0.2), in the case of Artin-Schreier extensions. The classical
theory considers the case of complete discrete valued field
extension $L|K$ where the residue field $k$ of $K$ is perfect. 
In his paper \cite{K2}, Kato gives a natural definition of the
Swan conductor for complete discrete valuation rings with arbitrary (possibly imperfect) residue fields. He also defines the refined Swan
conductor $\rsw$ in this case using differential 1-forms and
powers of the maximal ideal $\m L$. The generalization we
present is a further refinement of this definition. Moreover, we
can deal with the extensions with defect, a case which was not
treated previously. 

\subsection{ Invariants of Ramification Theory} 

Let $K$ be a valued field of characteristic $p>0$ with henselian
valuation ring $A$,  valuation $v_K$ and residue field
$k$. Let $L=K(\alpha)$ be the Artin-Schreier extension defined
by $\ap^p - \ap =f$ for some $f \in K^{\times} $. Assume that
$L|K$ is non-trivial, that is, $[L:K]=p$. Let $B$ be the
integral closure of $A$ in $L$. Since $A$ is henselian, it
follows that $B$ is a valuation ring. Let $v_L$ be the valuation
on $L$ that extends $v_K$ and let $l$ denote the residue field
of $L$. Let $\Gamma :=v_K(K^{\times})$ denote the value group of
$K$. The Galois group $\Gal(L|K)=G$ is cyclic of order $p$,
generated by $ \s : \ap \mapsto \ap +1$.\\

Let $\mathfrak{A} = \{f \in K^{\times} \ | $ the solutions of
the equation $\ap^p-\ap=f$ generate $L$ over $K \} $.
Consider the ideals $\js$ and $H$, of $B$ and $A$ respectively,
defined as below:
\begin{equation} 
\js = \ \lb \left\{ \frac{\s(b)}{b}-1 \mid b \in L^{\times}
\right\} \rb \subset B
\end{equation} 

\begin{equation}
H= \ \lb \left\{ \frac{1}{f} \mid f \in \mathfrak{A} \right\}
\rb \subset A
\end{equation}

Our first result compares these two invariants via the norm map
$N_{L|K}=N$, by considering the ideal $\ns$ of $A$ generated by
the elements of $N(\js)$.
We also consider the ideal $\dis \is = \lb \{ \s(b) -b \ |\ b
\in B \} \rb$ of $B$. The ideals $\is$ and $\js$ play the roles
of $i(\s)$ and $j(\s)$ (the Lefschetz numbers in the classical
case, as explained in 2.1), respectively, in the generalization.
\subsection{Main Results}
We did not make any assumptions regarding the rank or defect in
these definitions. Now consider two special cases of the
scenario described above:
\begin{enumerate}[(I)]
\item \textbf{(Defectless)} In this case, we assume that $L|K$
is defectless. For Artin-Schreier extensions $L|K$ considered in
this paper, it means that either
$v_L(L^{\times})/v_L(K^{\times})$ has order $p$ and the residue
extension $l|k$ is trivial or the residue extension $l|k$ is of
degree $p$ and $L$ has the same value group $\Gamma$ as $K$.
\item \textbf{(Rank $1$)} The value group $\Gamma$ of $K$ is
isomorphic to a subgroup of $\R$ as an ordered group.
\end{enumerate}
We will prove the following results:

\begin{thm}\label{hn} If $L|K$ satisfies (I) or (II), we have
the following equality of ideals of $A$:
\begin{equation}
   H = \ns
\end{equation}
\end{thm}
\begin{thm}\label{comm dia} If $L|K$ satisfies (I) or (II), we
consider the $A$-module $\ok$ of logarithmic differential
$1$-forms and the $B$-module $\ol$ of logarithmic differential
$1$-forms over $A$ (as defined in section 1.1). Then
\begin{enumerate}[(i)]
\item There exists a unique homomorphism of $A$-modules 
$\dis \rsw : H/H^2 \to \ok/(\is \cap A)\ok$ such that $\dis
\frac{1}{f} \mapsto \dl f \ ; \ $ for all $ f \in
\mathfrak{A}$.
\item There is a $B$- module isomorphism $\dis \varphi_\s:
\ol/\js\ol \overset{\cong}{\to }\js/\js^2$ such that $ \dis \ \dl x \mapsto \frac{\s(x)}{x}-1,$ for all $ x \in L^{\times}.$
\item Furthermore, these maps induce the following commutative
diagram:

\begin{center}$
\begin{tikzcd}[column sep=large]
\ol/\js \ol \arrow{r}{\varphi_\s}[swap]{\cong}
\arrow{d}[swap]{\dn}
&\js/\js^2  \arrow[hook]{d}{\overline{N}}\\
\ok/(\is \cap A)\ok   &H/H^2 \arrow{l}{\rsw}
\end{tikzcd}$
\end{center}

\noindent The maps $\dn, \overline{N}$ are induced by the norm
map $N$, as described in section 6.
\end{enumerate}
\end{thm}

The map $\rsw$ in (i) is a refined generalization of the refined
Swan conductor of Kato for complete discrete valuation rings
\cite{K2}.
\begin{rem} It is worth noting that if $p=2$, both the results
are true without any assumptions regarding defect or rank, as
seen in later sections.
\end{rem}
\begin{rem} If $L|K$ is unramified ($e_{L|K}=1, l|k$ separable of degree $p$), then we have $i(\s)=j(\s)=0$,  $\is=\js=B$ and $H=A$. Consequently, our main results are trivially true. From now on, we assume that $L|K$ is either wild ($e_{L|K}=p, l|k$ trivial ), ferocious ($l|k$ purely inseparable of degree $p$) or with defect.
\end{rem}
\subsection{Outline of the Contents} 

\begin{itemize}
\item \textbf{Review, Small Results, Examples:} In sections 1, 2
we present some preliminaries and the discrete valuation ring
case. Section 3 contains some elementary results that help us
understand the cases I and II. 
\item \textbf{Proofs of Main Results:} We prove  \Cref{hn} in section 4. In section 5, we analyze 
the defect case. We use \Cref{hn} to prove \Cref{fil}, which  allows us to express the ring $B$
as a filtered union of rings $A[x]|A$, where elements $x \in
L^{\times}$ are chosen very carefully.  We prove  
\Cref{comm dia} for both cases I and II in section 6.

\item \textbf{The Different Ideal and Further Results:} Section
7 presents the description of the different ideal $\D_{B|A}$
when $L|K$ satisfies (I) or (II). This ideal equals the annihilator of the relative K{\"a}hler differential module $\Omega^1_{L|K}$ in the classical case. However, this is not true in the case of arbitrary valuations.

\item \textbf{Appendix:} In section 8, we present a non-trivial example of a defect extension. It shows us the difficulties that rise from the defect. We also verify the main results for this example.
\end{itemize}

\section{Preliminaries: Differential Forms, Defect, Cyclic
Extensions, Trace}

\subsection{Definitions: Differential Forms and Different Ideal
$\D_{B|A}$ }
\begin{defn} \textbf{Differential $1$-Forms}

\begin{enumerate}[(i)]
\item Let $R$ be a commutative ring. The $R$-module $\Om_R^1$ of
\textit{differential $1$-forms over $R$} is defined as follows:
$\Om_R^1$ is generated by
\begin{itemize}
\item The set $\{ db \mid b \in R \}$ of generators.
\item The relations being the usual rules of differentiation:
For all $b, c \in R$,
\begin{enumerate}
\item (Additivity) $d(b+c)=db+dc$
\item (Leibniz rule) $d(bc)=cdb+bdc$
\end{enumerate}
\end{itemize}
\item For a commutative ring $A$ and a commutative $A$-algebra
$B$, the $B$-module $\OL$ of \textit{relative differential
$1$-forms over $A$} is defined to be the cokernel of the map $B
\otimes_A \Om^1_A \to \Om^1_B.$
\end{enumerate}
\end{defn}

\begin{defn} \textbf{Logarithmic Differential $1$-Forms}
\begin{enumerate}[(i)]
\item For a valuation ring $A$ with the field of fractions $K$,
we define the \textit{$A$-module $\ok$ of logarithmic
differential $1$-forms} as follows: $\ok$ is generated by
\begin{itemize}
\item The set $\{ db \mid b \in A \} \cup \{ \dl x \mid x
\in K^{\times}\}$ of generators.
\item The relations being the usual rules of differentiation and
an additional rule: For all $b,c \in A$ and for all $x,y \in
K^{\times},$
\begin{enumerate}
\item (Additivity) $d(b+c)=db+dc$
\item (Leibniz rule) $d(bc)=cdb+bdc$
\item (Log 1) $\dl(xy)=\dl x+\dl y$
\item (Log 2) $b \dl b = db$ for all $0 \neq b \in A$
\end{enumerate}
\end{itemize}
\item Let $L|K$ be an extension of henselian valued fields, $B$
the integral closure of $A$ in $L$ and hence, a valuation ring.
We define the \textit{$B$-module $\ol$ of logarithmic relative
differential $1$-forms over $A$} to be the cokernel of the map
$B \otimes_A \om^1_A \to \om^1_B.$

\end{enumerate}
\end{defn}
\begin{defn} \textbf{The Different Ideal $\D_{B|A}$}
 
Let $A$ be a valuation ring with the field
of fractions $K$. Let $L|K$ be a separable extension of fields,
$B$ the integral closure of $A$ in $L$.
As in the classical case, we define the \textit{inverse
different} $\D_{B|A}^{-1}$ by $\dis \D_{B|A}^{-1}:= \{x \in L \
| \ \Tr_{L|K}(xB) \subset A \}$.

This is a fractional ideal  of $L$.
The \textit{different $\D_{B|A}$ of $B$ with respect to $A$} is
defined as the inverse ideal of $\D_{B|A}^{-1}$.\\
\end{defn}

\subsection{Valuation Rings and Differential $1$-Forms}
\begin{defn} Let $A$ be a valuation ring with fraction field $K$
and  valuation $v$. For any $x \in K^{\times}$, we
define an $A$-module homomorphism $\dis dx: M_x \to \ok$ by $h
\mapsto hx \ \dl x$ where $M_x:=\lb \frac{1}{x} \rb$.\\
For $x=0$, we define $d0$ to be the zero map$: M_0 \to \ok$ by
$h \mapsto 0$ where $M_0:=K$.
\end{defn}
\begin{lm} Let $A, K, v$ be as above and $x, y \in K$. Then we
have the following properties.
\begin{enumerate}[(i)]
\item (Additivity) The $A$-module homomorphisms $dx, dy, d(x+y):
M \to \ok$ satisfy $d(x+y)=dx+dy$. Here, $M= M_x \cap M_y \cap
M_{x+y} $.
\item (Leibniz rule) The $A$-module homomorphisms $dx, dy,
d(xy): M \to \ok$ satisfy $d(xy)=ydx+xdy$. Here, $M=M_x \cap M_y
\cap M_{xy} $.
\end{enumerate}
\end{lm}
\begin{proof} \begin{enumerate}[(i)] \item We may assume that
$v(x) \leq v(y)$ and write $y=ax;\ a \in A$. Note that in $\ok,
da=a \dl a$ and $d1=\dl 1 =0$. Hence, $(a+1) \dl
(a+1)=d(a+1)=da=a \dl a$.\\
For all $ h \in M$,
\begin{align*}
d(x+y)(h)&= h(x+y) \dl (x+y)\\&=hx(a+1) \dl [x(a+1)]\\
&=hx(a+1) [\dl x + \dl (a+1)]\\&= hx(a+1) \ \dl x +
hx(a+1) \ \dl (a+1)\\&=
hx \ \dl x + hxa \ \dl x + hxa \ \dl a \\&= hx \ \dl
x + hxa \ \dl xa\\&= dx(h) + dy(h)
\end{align*}
\item For all $ h \in M$,
\begin{align*}
d(xy)(h)&= hxy\ \dl (xy)\\&= hxy \ \dl x +hxy \ \dl y
\\&= y dx(h)+xdy(h)
\end{align*}
\end{enumerate}
\end{proof}

\begin{lm}\label{phi} Let $L|K$ be as in 0.1. Then we have 
\begin{enumerate}
\item A surjective $B$-module homomorphism $\Phi_\s : \OL/\is\OL
\to \is/\is^2$ such that $\Phi_\s(db)= \s(b)-b$ for all $b \in
B$.
\item A surjective $B$-module homomorphism $\varphi_\s :
\ol/\js\ol \to \js/\js^2$ such that $\varphi_\s( \dl x)=
\frac{\s(x)}{x} -1$ for all $x \in L^{\times}.$
\end{enumerate}
\end{lm}
\begin{proof} Since $\s$ fixes $K$, $\s(a)-a=0$ for all $a \in
A$ and $\frac{\s(x)}{x}-1=0$ for all $x \in K^{\times}$.
Let $b,c \in B$.
The first part follows from $\dis  \s(b+c)-(b+c)=\s(b)-b+\s(c)-c$ and
 \begin{align*}
  \s(bc)-bc &=(\s(b)-b)(\s(c)-c)+c(\s(b)-b)+b(\s(c)-c)\\
&\equiv c(\s(b)-b)+b(\s(c)-c)
\mod \is^2. 
\end{align*}

\noindent Let $x,y \in L^{\times}$. The second assertion follows from
\begin{align*}
\frac{\s(xy)}{xy}-1&=\lb \frac{\s(x)}{x}-1\rb \lb
\frac{\s(y)}{y}-1 \rb+ \frac{\s(x)}{x}-1+\frac{\s(y)}{y}-1 \\ &
\equiv \frac{\s(x)}{x}-1+\frac{\s(y)}{y}-1 \mod \js^2.
\end{align*}

\end{proof}

\subsection{Defect: Introduction}
 
\begin{defn}
Let $E|F$ be a finite algebraic extension of fields of degree
$[E:F]=n$ and $v$ a non-trivial valuation on $F$. Denote the
extensions of $v$ from $F$ to $E$ by $v_1,...v_g$. Let $F_v$ be
the residue field and $v(F^{\times})$ the value group for the
valued field $(F,v)$. Similarly, define $E_{v_i
}$ and $v_i(E^{\times})$.  For each $1 \leq i \leq g$, define:
\begin{itemize}
\item The ramification index $e_i=
(v_i(E^{\times}):v(F^{\times}))$
\item The inertia degree $f_i=[E_{v_i} :F_v]$.
\end{itemize}
\end{defn}

\noindent \textbf{Fact I:} For each $1 \leq i \leq g$, $e_i$ and
$f_i$ are finite. Moreover, we have the \textbf{fundamental
inequality}:
\begin{equation}
[E:F]=n \geq \sum^g_{i=1} e_if_i
\end{equation}
If the equality holds, it is called the \textbf{fundamental
equality}.\\

\noindent \textbf{Fact II:} When $(F,v)$ is henselian, $g=1$ and
we deal with a single ramification index $e_{E|F}=e$ and a
single inertia degree $f_{E|F}=f$. Furthermore, in this case,
$n$ is divisible by the product $ef$ and we can write
\begin{equation}
n=d_{E|F}e_{E|F}f_{E|F}
\end{equation}
for some positive integer $d_{E|F}$.
\begin{defn} The integer $d_{E|F}$ above is called the
\textit{defect} of the extension $(E|F,v)$. It is known that
$d_{E|F}$ is a power of $q$ ; where $q= \max \{\ch (F_v),1\}$.
\end{defn}
 
\subsection{Cyclic Extensions of Prime Degree} 
Let $E|F$ be a cyclic degree $p$ Galois extension of henselian valued
fields, where $p=\ch{F} >0$ . Let $\mathcal{O}_E$ and
$\mathcal{O}_F$ denote the valuation rings of $E$ and $F$
respectively. Let $\overline{E}$ and $\overline{F}$ be the
respective residue fields.

\begin{lm}\label{cyclic}
If $E|F$ is ramified and defectless, then we have two cases:

\begin{enumerate}[(a)]
\item Order of $v(E^{\times}) / v(F^{\times})$ is $p$ and it is
generated by $v_E(\mu)$ for some $\mu \in E^{\times}$.
\item There is some $\mu \in E^{\times}$ such that the residue
extension $\overline{E}|\overline{F}$ is purely inseparable of
degree $p$, generated by the residue class of $\mu$.
\end{enumerate}
\end{lm} 

\begin{lm}\label{mu}
Let $E|F, \mu$ be as in \Cref{cyclic} and $x_i \in F$ for all $0
\leq i \leq p-1$ . Then

$\dis \sum_{i=0}^{p-1} x_i \mu^{i} \in \mathcal{O}_E$ if and
only if $\dis x_i \mu^i \in \mathcal{O}_E $ for all $ i$.
\end{lm}
\begin{proof} If $\dis x_i \mu^i \in \mathcal{O}_E $ for all $
i$, then clearly, $\dis \sum_{i=0}^{p-1} x_i \mu^{i} \in
\mathcal{O}_E$ . For the converse, we observe that if $v_E(x_i
\mu^{i})$ are all distinct for $0 \leq i \leq p-1$, then $\dis 0
\leq v_E \lb \sum_{i=0}^{p-1} x_i \mu^{i} \rb= \min_{0 \leq i
\leq p-1}v_E (x_i \mu^{i})$. Hence, the converse is true in this
case. Now let us break down the rest into two cases (a) and (b)
as described in the lemma above.\\
\begin{enumerate}[(a)]
\item We claim that in this case, $\dis v_E(x_i \mu^{i}); 0 \leq
i \leq p-1, \ x_i \neq 0$ all have to be distinct.\\
Assume to the contrary. Let $0 \leq i <j \leq p-1$ be such that
$v_E(x_i\mu^i)=v_E(x_j\mu^j); \ x_i, x_j$ are non-zero. Then
$\dis v_E(\mu^{j-i})=(j-i)v_E(\mu)=v_E \lb \frac{x_i}{x_j}\rb
\in v(F^{\times})$. This is impossible, since the order of
$v_E(\mu)$ in $v(E^{\times}) / v(F^{\times})$ is $p$ and $p
\nmid j-i$.
\item We observe that $v(\mu)=0$. The only case we need to
consider is when $\dis \min_{0 \leq i \leq p-1}v(x_i
\mu^{i})=v<0$ and the minimum is achieved by more than one
$x_i\mu^i$. Let $0\leq i_1<...<i_r \leq p-1 ; r \geq 2$ integer
such that $v(x_{i_s} \mu^{i_s})=v$ for all $1 \leq s \leq r$.
Clearly, $v(x_{i_s})=v$ for all $1 \leq s \leq r$. In
particular, $x_{i_1} \neq 0$. Since $\dis v \lb
\sum_{s=1}^rx_{i_s} \mu^{i_s}\rb>v$, we see that $\dis v \lb
\sum_{s=1}^r\frac{x_{i_s}}{x_{i_1}} \mu^{i_s} \rb>0$.

\noindent Equivalently, $\dis
z=\sum_{s=1}^r\frac{x_{i_s}}{x_{i_1}} \mu^{i_s} \in \m E$; where
$\m E$ is the maximal ideal of $\mathcal{O}_E$.

\noindent Since $\overline{\mu}^i$'s are $\overline{F}$-linearly
independent ; $\dis 0 \leq i \leq p-1$, $z \in \m E \Lr
\overline{z} = 0 \in \overline{E} \Lr \overline{\lb
\frac{x_{i_s}}{x_{i_1}}\rb }=0 \in \overline{F} \ $ for all $
s.$ However, this is impossible since $v(x_{i_s})=v$ for all $1
\leq s \leq r$.
\end{enumerate}
\end{proof}

\begin{lm}\label{dmu} 
 Let $\mu$ be as in \Cref{cyclic}. Then 
$\dl \mu$ generates the $\mathcal{O}_E$-module
$\om^1_{\mathcal{O}_E|\mathcal{O}_F}$.
\end{lm}
\begin{proof} It is enough to consider the
elements $\dl (x \mu^i) ; 0 \leq i \leq p-1, \ x \in K^{\times} $.

\noindent $\dl (x \mu^i)= \dl x + i \ \dl \mu$. The
rest follows from the fact that $\dl x = 0$ in
$\om^1_{\mathcal{O}_E|\mathcal{O}_F}$.

\end{proof}

\subsection{Trace}

\begin{lm}\label{trace}
Let $R$ be an integrally closed integral domain with the field
of fractions $F$. Let $E|F$ be a separable extension of fields of
degree $n$. Suppose that $\bt \in E$ is such that $E=F(\bt)$ .
Let $g(T)= \min_F(\bt)$, the minimal polynomial of $\bt$ over
$F$. Then

\begin{enumerate}
\item $\Tr_{E|F}\lb \frac{\bt^m}{g'(\bt)}\rb$ is zero for all $1
\leq m \leq n-2$ and $\Tr_{E|F}\lb \frac{\bt^{n-1}}{g'(\bt)}\rb
=1$.
\item Assume, in addition, that $\bt$ is integral over $R$. Then
$\dis \{ x \in E \mid \Tr_{E|F}(xR[\bt]) \subset R \} =
\frac{1}{g'(\bt)}R[\bt]$
\end{enumerate}
\end{lm}
\noindent Details can be found in section 6.3 of \cite{K}.

\section{Discrete Valuation Rings}
\subsection{Classical Theory: Complete Discrete Valuation Rings
with Perfect Residue Fields}
Let $K$ be a complete discrete valued field of residue
characteristic $p>0$ with normalized  valuation $v_K$,
valuation ring $A$ and perfect residue field $k$. Consider
$L|K$, a finite Galois extension of $K$. Let $e_{L|K}$ be
the ramification index of $L|K$ and $G= \Gal (L|K)$. Let $v_L$ be
the valuation on $L$ that extends $v_K$, $B$ the integral
closure of $A$ in $L$ and $l$ the residue field of $L$. In this
case, we have the following invariants of ramification theory:
\begin{itemize}
\item The Lefschetz number $i(\s)$ and the logarithmic Lefschetz
number $j(\s)$ for $\s \in G\backslash \{1 \}$ are defined as
\begin{equation}
i(\s)=\min\{v_L(\s(a)-a)\mid a \in B \}
\end{equation}
\begin{equation}
j(\s)=\min\{v_L(\frac{\s(a)}{a} -1)\mid a \in L^{\times} \}
\end{equation}
Both the numbers are non-negative integers.
\item For a finite dimensional representation $\rho$ of $G$ over
a field of characteristic zero, the Artin conductor $\Art(\rho)$
and the Swan conductor $\Sw(\rho)$ are defined as
\begin{equation}
\Art(\rho)= \frac{1}{e_{L|K}} \sum_{\s \in G\backslash \{1 \}}
i(\s)(\dim(\rho)-\Tr(\rho(\s)))
\end{equation}
\begin{equation}
\Sw(\rho)= \frac{1}{e_{L|K}} \sum_{\s \in G\backslash \{1 \}}
j(\s)(\dim(\rho)-\Tr(\rho(\s)))
\end{equation}
Both these conductors are integers. This is a consequence of the
Hasse-Arf Theorem (see \cite{S} ).
\end{itemize}

\noindent The invariants $j(\s)$ and $\Sw(\rho)$ are the parts of
$i(\s)$ and $\Art(\rho)$, respectively, which handle the wild
ramification. We wish to generalize these to all valuation rings
considered in this paper. Namely, the case where $L$ is a
non-trivial Artin Schreier extension of $K$, a valued field with
henselian valuation ring, defined by $\ap^p - \ap =f$, where $f
\in K$ . Let us begin with the case of discrete valuation rings,
possibly with imperfect residue fields.

\subsection{ Best $f$ and Swan Conductor}
Let $K$ be a complete discrete valued field of residue characteristic
$p>0$ with normalized  valuation $v_K$, valuation ring
$A$ and residue field $k$. We do not assume that $k$ is perfect.
Let $L=K(\alpha)$ be the (non-trivial) Artin-Schreier extension
defined by $ \ap^p - \ap =f$, where $f \in K$. Let $v_L$, $B$
and $l$ denote the valuation, valuation ring and the residue
field of $L$, respectively. We define the Swan conductor of this
extension as described below.

\begin{defn}Let $\mathfrak{P} : K \to K $ denote the additive
homomorphism $x \mapsto x^p-x$. Note that the extension $L$ does
not change when $f$ is replaced by any element $g \in K$ such
that $g \equiv f$ mod $ \mathfrak{P} (K)$. Because, if
$g=f+h^p-h$ for some $h \in K$, then the corresponding
Artin-Schreier extension is generated by $\ap +h$ over $K$.

\begin{enumerate}
\item If there is such $g \in A$, $L$ is unramified over $K$ and
the Swan conductor is defined to be $0$.
\item If there is no such $g \in A$, the Swan conductor is
defined to be $\min \{-v_K(g) \mid g \equiv f \mod \mathfrak{P}
(K) \}$. An element $f$ of $K$ which attains this minimum will
be referred to as `` best $f$ " throughout this paper. It is
well-defined modulo $\mathfrak{P} (K)$.
\end{enumerate}
This definition coincides with the classical definition of the
Swan conductor when $k$ is perfect. \\\\Existence of best $f$
relies on the existence of $\min \{-v_K(g) \mid g \equiv f \mod
\mathfrak{P} (K) \}$. This is guaranteed in the case of discrete
valuation rings, but not in the case of general valuation rings.
\end{defn}
\begin{ex} Let $K=k((t))$ where $k$ is of characteristic $p>0$.
$t $ is a prime element of $K$. Let $n$ be a positive integer
coprime to $p$. In this case, the Swan conductor of the
extension given by $\dis \ap^p-\ap= \frac{1}{t^{n}}$ is $n$.

More generally, let $m \geq 0$ be an integer and $n$ as above.
Then the Swan conductor of the extension given by $\dis
\ap^p-\ap= \frac{1}{t^{np^m}}$ is also $n$. This follows from
$\dis \frac{1}{t^{np^m}}\equiv \frac{1}{t^{n}}$ mod $
\mathfrak{P}(K)$.
\end{ex} 
\noindent A concrete description of the Swan conductor is given by the
following lemma:
\begin{lm}
By replacing $f$ with an element of $\{g \in K \mid g \equiv f
\mod \mathfrak{P} (K) \}$, we have best $f$ which satisfies
exactly one of the following properties:
\begin{enumerate}[(i)]
\item $f \in A$
\item $v_K(f)=-n$ where $n$ is a positive integer relatively
prime to $p$.
\item $f=at^{-n}$ where $n>0$, $p|n$, $t$ is a prime element of
$K$ and $a \in A^{\times}$ such that the residue class of $a$ in
$k$ does not belong to $k^p=\{x^p \mid x \in k \}$.
\end{enumerate}
In the case (i), the Swan conductor is 0. In the cases (ii) and
(iii), the Swan conductor is $n$.
\end{lm}

\subsection{Refined Swan Conductor $\rsw$}
\begin{defn}
Let $K$ be a discrete valued field of residue characteristic
$p>0$ with normalized  valuation $v_K$, valuation ring
$A$ and residue field $k$ (possibly imperfect). Let
$L=K(\alpha)$ be the Artin-Schreier extension defined by $\ap^p
- \ap =f$ where $f$ is best. The \textit{refined Swan conductor
($\rsw$)} of this extension is defined to be the $A$-homomorphism
$\dis df: \lb \frac{1}{f} \rb \to \ok$ given by $ h \mapsto
(hf)\ \dl f$. We note that for $\dis h \in \lb \frac{1}{f}
\rb, hf \in A$ and hence, $(hf)\ \dl f$ is indeed an element
of $\ok$.

The $A$-homomorphism $\rsw$ is well-defined up to certain relations, as discussed
below.
\end{defn}
\begin{lm} Let $L|K$ be as above, given by best $f$, $H=\lb
\frac{1}{f} \rb$. Then $\rsw$ is well-defined as the
$A$-homomorphism $: H \to \ok / \I \ok$; where $\I$ is the ideal
$\{x \in K \mid v_K(x)\geq \lb \frac{p-1}{p} \rb v_K \lb
\frac{1}{f} \rb \}$ of $A$.
\end{lm}
\begin{proof} Let $g$ be best as well. Hence, there exists $ a
\in K$ such that $g =f + a^p -a$ and $v_K(f+a^p-a)=v_K(f)$.
Since $v_K(a) \geq v_K(f), H \cap M_a=H$. By Lemma 1.5,
$dg-df=-da$ on $H$.

For $\dis h=\frac{b}{f} \in H; b \in A, da(h)=ha\ \dl a=b \lb
\frac{a}{f} \rb \dl a \in \lb \frac{a}{f} \rb \ok$. It is
enough to show that $\dis v_K \lb \frac{a}{f}\rb \geq \lb
\frac{p-1}{p} \rb v_K \lb \frac{1}{f} \rb $. This is clear in
the case $a \in A$.

If $a \in K \backslash A$, then $ v_K(a^p-a)=pv_K(a) \geq
v_K(f)=v_K(f+a^p-a)$. Hence, proved.
\end{proof}

\begin{rem} We note that $\I =\{x \in K \mid v_K(x)\geq \lb
\frac{p-1}{p} \rb v_K \lb \frac{1}{f} \rb \}=\{x \in K \mid
v_L(x)\geq \lb \frac{p-1}{p} \rb v_L \lb \frac{1}{f} \rb \}$
\end{rem}

\section{Small Results}
In this section, we present some small results that help us
understand the two special cases I and II better. First we
extend the notion of ``best $f$" to the general case.
\subsection{Best $f$}
\begin{defn} Let $K$ be as in 0.1, $\mathfrak{P}:K \to K$ as
before. We say that $f \in K^{\times}$ is best if either $f \in
A^{\times}$ or if $f$ satisfies $-v(f)=\inf \{-v(g) \mid g
\equiv f \mod \mathfrak{P}(K) \}$.

\end{defn}

Since we cannot guarantee the existence of best $f$ in general,
as seen in the example below, we will reinterpret the notion of
the refined Swan conductor using the logarithmic differential
$1$-forms over $A$, as stated in \Cref{comm dia}.

\begin{ex} \textbf{(Non-DVR)}

Consider the extension $L|K$ as described in section 8. The
value group $\Gamma$ is isomorphic to $\Z[\frac{1}{p}]$. We have
a sequence of elements $f_i \in \mathfrak{A}$ for all
integers $i \geq 0$, each better than the previous one, such
that

\begin{enumerate}
\item $\dis  -v(f_i)= n- \sum_{j=1}^i \frac{1}{p^j}$ 
\item The ideal $H$ of $A$ is generated by $\{ \frac{1}{f_i}\ |
\ i \geq 0 \}$.
\end{enumerate}
Since
$\dis \inf_{i \geq 0} -v(f_i)=n - \frac{1}{p-1}=c \in \R
\backslash \Gamma,$ there is no best $f$.
\end{ex}
\begin{cor}\label{dlogalpha}
Let $K$ be as in 0.1 and $L|K$ satisfy (I), given by $\ap^p-\ap=f$ where $f$ is best.
Then
\begin{enumerate}[(i)]
\item $B$ is described as follows: 
\begin{enumerate}[(a)]
\item If $\dis e_{L|K}=p , B= \sum_{i=0}^{p-1} A_i\ap^i$ where
$A_0:=A$ and for all $1 \leq i \leq p-1$, \\$A_i := \{ x \in K \
| \ v_L(x) \geq -iv_L(\ap) \}= \{ x \in A \mid v_L(x) >
-iv_L(\ap) \}$ .
\item If $\dis e_{L|K}=1 , B= A[\ap \g]$ where $\g \in A$ such
that $\ap \g \in B^{\times}$.
\end{enumerate}
\item $\dis \dl \ap$ generates the $B$-module $\dis \ol$.
\end{enumerate}
\end{cor}
\begin{proof} \begin{enumerate}[(i)]
\item We apply \Cref{mu}
\begin{enumerate}[(a)]
\item $-v_0:=v_L(\ap)$ generates the group $v_L(L^{\times}) /
v_L(K^{\times})$ of order $p$. In particular, for all $1 \leq i
\leq p-1, \\iv_0 \notin v_L(K^{\times})$. For $x \in
K^{\times}$, $x\ap^i \in B$ if and only if $v_L(x) \geq iv_0$ if
and only if $v_L(x) > iv_0$.
\item Since $e=1$, there exists $\g \in A$ such that $\ap \g \in
B^{\times}$. We just take $\mu= \ap\g$.
\end{enumerate}
\item This  is a direct consequence of (i) and \Cref{dmu}.
\end{enumerate}
\end{proof}
\subsection{Fractional Ideals in a Valued Field:} 

Let $F$ be a valued field with  valuation $v$, value
group $\Gamma$, valuation ring $\mathcal{O}:=\mathcal{O}_F$ and
residue field $\overline{F}$. A subset $S$ of $F$ is a
\textit{fractional ideal} of $F$ if there exists $0 \neq b \in
\mathcal{O}$ such that $bS$ is an (integral) ideal of
$\mathcal{O}$.

We note that in such a case, $\dis S= \{x \in F \ |\ v(x) \geq
v(s)$ for some $s \in S \}= \cup_{s \in S} \ s \mathcal{O}$.

\begin{defn} Consider the case (II), we can regard $\Gamma$ as
an ordered subgroup of $\R$. Let $S$ be a fractional ideal of
$F$ and $\dis \inf_{s \in S} v(s)= t \in \R$. We define
\textit{$F$-valuation of $S$} as follows: \begin{enumerate}[(i)]
\item If $t \in \Gamma \subset \R, v(S):= t$
\item If $t \in \R \backslash \Gamma, v(S):=t^{+}$
\end{enumerate}

We can define the \textit{$F$-valuation of $S$} by (i) when $S$
is generated by a single element $s \in F$, even if $\Gamma$ is
not isomorphic to an ordered subgroup of $\R$. In that case,
$v(S):=v(s)$ and $S=s' \mathcal{O}$ for any $s' \in F$ such that
$v(s')=v(s)$.

\end{defn}

\subsection{Defect and $\js$}
\begin{lm}\label{jh} The fractional ideals $\js$ and $H$ are
integral ideals of $L$ and $K$ respectively, that is,
\begin{enumerate}[(i)]
\item $ \js = \ \lb \{ \frac{\s(b)}{b}-1 \mid b \in L^{\times}
\} \rb \ \subset B$
\item $ H= \ \lb \{ \frac{1}{f} \mid f \in \mathfrak{A} \} \rb
\ \subset A$
\end{enumerate}
\end{lm}
\begin{proof} \begin{enumerate}[(i)]
\item For $b \in L^{\times}$, $v_L(\s(b)-b) \geq \min
\{v_L(\s(b)), v_L(b)\}=v_L(\s(b))=v_L(b)$. Hence,
$\frac{\s(b)}{b}-1 \in B$.
\item We need to show that for each $\dis f \in \mathfrak{A},
\frac{1}{f} \in A$. Assume to the contrary that there is some $f
\in \m K \cap \mathfrak{A}$. Since $K$ is henselian, roots of
$\ap^p-\ap=f$ are already in $K$, contradicting our assumption
that $L|K$ is non-trivial.
\end{enumerate}
\end{proof}
We use the following lemma in order to prove 
\Cref{Di}, which plays a very important role in the proof of \Cref{j}.
\begin{lm}\label{sig} Let $L|K$ be as in 0.1 and $b \in B$ such
that $\s(b)-b$ generates $\is$. Define $A$-linear maps $D_i: L
\to L$ inductively for $0 \leq i \leq p-1$ by \begin{equation}
D_0:=id_L : L \to L , \ D_i(x):=
\frac{(\s-1)(D_{i-1}(x))}{(\s-1)(D_{i-1}(b^i))}\ ; 1 \leq i \leq
p-1 \end{equation}
These maps have the following properties:
\begin{enumerate}
\item $\dis D_i(b^i)=1; \ 0 \leq i \leq p-1$
\item $\dis D_i(b^j)=0 ;\ 0 \leq j \leq i-1, \ 1 \leq i \leq
p-1$
\item For $\dis x \in B, \ D_i(xb)=\s^i(b)D_i(x)+D_{i-1}(x); \ 0
\leq i \leq p-1$ (If $i=0$, we set $D_{i-1}(x)=0$.)
\item $\dis D_i(b^{i+1})= \sum_{j=0}^{i} \s^j(b); \ \ 0 \leq i
\leq p-1$
\item For each $0 \leq i \leq p-2$, $\dis
(\s-1)(D_i(b^{i+1}))=\s^{i+1}(b)-b$ and hence, is a generator of
$\is$.
\\In particular, it is non-zero.
\end{enumerate}
\end{lm}
\begin{proof} First we note that $(\s-1)(D_0(b^1))=(\s-1)(b)
\neq 0$ and hence, the definition of $D_1$ is valid. As we prove
(1)-(5) by induction on $i$, validity of the definition of $D_i$
for $ \ 1 \leq i \leq p-1$ will become clear.

(1) follows directly from the definition. (2) is clearly true
for $i=1$, since $D_1(1)=0$. \\If $0 \leq j \leq i-1 \leq p-2,
(\s-1)(D_{i-1}(b^j))= (\s-1)(1)$ or $(\s-1)(0)$ and hence,
$D_i(b^j)=0$.

\noindent The $i=0$ case of (3)-(5) follows directly from the
definition.\\ For $i=1$, (3)-(5) follow from
\begin{align*}
D_1(xb)&=\frac{(\s-1)(D_0(xb))}{(\s-1)(D_0(b))}=\frac{(\s-1)(xb)}{(\s-1)(b)}\\&=\frac{(\s-1)(x)
\s(b)+x (\s-1)(b)}{(\s-1)(b)}\\&=\s(b) D_1(x)+x=\s(b)
D_1(x)+D_0(x)
\end{align*}

\noindent Let $2 \leq i \leq p-1$ and assume that (3)-(5) are
true for $0, ..., i-2, i-1$. Then we have:

\begin{align*}
D_i(xb)&=\frac{(\s-1)(D_{i-1}(xb))}{(\s-1)(D_{i-1}(b^i))}&\\
&=\frac{(\s-1)(\s^{i-1}(b)D_{i-1}(x)+D_{i-2}(x))}{\s^i(b)-b} & \
by \ (3). \\
&=\frac{(\s-1)(D_{i-1}(x)).\s^i(b)+(\s-1)(\s^{i-1}(b)).D_{i-1}(x)+(\s-1)(D_{i-2}(x))}{\s^i(b)-b}&\\
&=\s^i(b)D_i(x)+\frac{(\s-1)(\s^{i-1}(b)).D_{i-1}(x)+(\s-1)(D_{i-2}(x))}{\s^i(b)-b}&\
by \ (3.7) \ and \ (5).\\
&=\s^i(b)D_i(x)+\frac{(\s-1)(\s^{i-1}(b)).D_{i-1}(x)+(\s-1)(D_{i-2}(b^{i-1})).D_{i-1}(x)}{\s^i(b)-b}&
\ by \ (3.7) \ for \ i-1.\\
&=\s^i(b)D_i(x)+D_{i-1}(x)\frac{(\s-1)(\s^{i-1}(b))+\s^{i-1}(b)-b}{\s^i(b)-b}&
\ by \ (5) \ for \ i-2.\\
&=\s^i(b)D_i(x)+D_{i-1}(x)&
\end{align*}
This proves (3) for $i$. (4) follows from (3). For any fixed $0
\leq i \leq p-2, (\s-1)(D_{i}(b^{i+1}))$ has the same valuation
as $(\s-1)(b)$ and hence generates $\is$.
\end{proof}
\begin{cor}\label{Di} For all $\dis 0 \leq i \leq p-1, D_i(B)$ is a subset
of $B$.
\end{cor}
\begin{proof} This is clearly true for $i=0$. We proceed by
induction. Fix some $\dis 1 \leq i \leq p-1$ and assume that the
statement is true for $i-1$. Hence, $ $ for all $ \ x \in B,
D_{i-1}(x) \in B \Ra (\s-1)(D_{i-1}(x)) \in \is$. By \Cref{sig}
(5) , $\dis D_i(x)=
\frac{(\s-1)(D_{i-1}(x))}{(\s-1)(D_{i-1}(b^i))} \in B \ $.
\end{proof}
\begin{lm}\label{ij} If $L|K$ is as in 0.1 and has defect, then
\begin{enumerate}[(i)]
\item $\dis \lb \s(b)- b \mid b \in B^{\times}\rb =\is = \js =
\lb \frac{\s(b)}{b}-1 \mid b \in B^{\times} \rb $
\item $ \OL = \ol $
\end{enumerate}
\end{lm}
\begin{proof} Given any $ b \in L$, there are elements $ a \in
K, \ b' \in B^{\times}$ such that $b=ab'$.
\begin{enumerate}[(i)]
\item For $b \in B, a \in A$ and $\s(b)- b=a(\s(b')- b' )$. 

\noindent For $b \in L^{\times}, a \in K^{\times}$ and
$\frac{\s(b)}{b}-1=\frac{a\s(b')}{ab'}-1=\frac{\s(b')}{b'}-1$.

\noindent Furthermore, $b' \in B^{\times} \Ra
\frac{\s(b')}{b'}-1=\frac{1}{b'}(\s(b')-b') \in \is$.
\item Let $b \in L^{\times}$. $\dl b=\dl a + \dl b'= \dl b'$ since $\dl a =0$ in $\ol$.

\noindent $b' \in B^{\times} \Ra \dl b' = \frac{1}{b'} db'
\in \OL$.
\end{enumerate}
\end{proof}
\begin{pr}\label{j} Let $L|K$ be as in 0.1.
$\js$ is principal if and only if $L|K$ is defectless.
\end{pr}
\begin{proof}
If the extension is defectless, by \Cref{cyclic}, \Cref{dmu} and
\Cref{phi}(a) $\js$ is principal.
Now suppose that the extension is with defect and that $\js$ is
principal. Hence, by \Cref{ij} $\js = \is$. Let $b\in B$ such
that $\s(b)-b$ generates $\is$.

\noindent We claim that $B = A[b]$.\\ Consider $\dis x_i \in K ;
0 \leq i \leq p-1$ such that $\dis y= \sum_{i=0}^{p-1}x_i b^i
\in B$. We must show that $x_i \in A; \ $ for all $ i$. \\
Define $y_i:=\sum_{j=0}^{p-i}x_j b^j; 1 \leq i \leq p$. We show
that $y_i \in B$ and consequently, $ D_{p-i}(y_i)=x_{p-i} \in
B$ by \Cref{Di}. Clearly, $y_1=y \in B$. Assume that $y_i \in B$ for some $1
\leq i \leq p$. Since $x_{p-i}, b \in B$, $y_{i+1}=y_i -
x_{p-i}b^{p-i} \in B$. Thus, $x_i \in B \cap K= A; \ $ for all $
i$ and hence, $B = A[b]$.\\

Since the extension is with defect, $f=1$ and $b = a+b'$ for
some $a \in A$ and for some $ b' \in m_L$. Therefore, we may
assume $b \in m_L$.
Also, due to the defect, $e=1$ and $b=ab'$ for some $a \in m_K$
and for some unit $ b' $ of $B$.
$\s(b)-b=a (\s(b')-b')$. $\s(b')-b' =c(\s(b)-b)$ for some $c \in
B$. Hence, $ac=1$. This is impossible since $a \in m_K$. Thus,
the extension must be defectless if $\js$ is principal.\\
\end{proof}
\vspace{20pt}

\section{Proof of \Cref{hn}}

We prove that $H=\ns$.

Let $ f \in \mathfrak{A}$. Then $ (-1)^pN(\ap)=-f$.
Equivalently, $
\frac{1}{f}=N(\frac{1}{\ap})=N(\frac{\s(\ap)}{\ap}-1)$. From
this, it follows that $H$ is a subset of $\ns$, without any
assumptions regarding defect or the value group $\Gamma_K$.
Next, we prove the reverse inclusion $\ns \subset H$. If $L|K$
is defectless, this follows directly from results in section 3. Because, 
$H$ is generated by $\frac{1}{f}$, where $f$ is best. Since
$\js=\lb \frac{1}{\ap} \rb,\ \ns=\lb N(\frac{1}{\ap}) \rb=H$.
Proof in the defect case, however, requires some work.

Let $L|K$ satisfy (II) and have defect. The value group
$\Gamma=\Gamma_K$ can be regarded as an ordered subgroup of
$\R$. Let $v$ denote the valuation on $L$ and also on $K$. We
analyze a special case first.

\subsection{Case $p=2$}  

For any $x\in L$, $\s(\s(x)-x)=x-\s(x)=\s(x)-x$, since the
characteristic is 2. Hence, $\s(x)-x= \s(x)+x=\Tr_{L|K}(x) \in K;
$ for all $ x \in L$. For any fixed $x \in L$, let $y=\s(x)-x$.
$\s(\frac{x}{y})-\frac{x}{y}=1$ if $y$ is non-zero, that is, if $x$
does not belong to $K$. Let $z=\frac{x}{y}$. $z+ \s(z) =2z+1=1$
and $N(z)=z(z+1)=z^2+z=z^2-z=f \in K$. Thus, $\frac{x}{y}$ is a
solution of an Artin-Schreier extension $ \ap^2-\ap =f ; f \in
K$. All Artin-Schreier extensions over $K$ having solution in $L$
are obtained in this way.

Any generator of $\js $ has the form $\frac{\s(x)-x}{x}$.
Letting $\frac{1}{f}=N(\frac{\s(x)-x}{x})$ we get the
corresponding Artin-Schreier extension.

\begin{rem} We don't need $\Gamma$ to be an ordered subgroup of
$\R$ for this case, the argument is true for any value group.
\end{rem}

\subsection{Case $p>2$} We wish to show $ \ns \subset H$,
equivalently, for each $\bt \in L^{\times} \backslash
K^{\times}$ the ideal of $A$ generated by
$N(\frac{\s(\bt)}{\bt}-1)$ is a subset of $H$.

\noindent Let us begin with some elementary observations:

\begin{enumerate}[(O1)]

\item \textbf{We may assume $\bt \in B \backslash A$:}
$\frac{\s(1/\bt)}{1/\bt}-1=(\frac{\s(\bt)}{\bt}-1)
(-\frac{\bt}{\s(\bt)})$. Since $(-\frac{\bt}{\s(\bt)}) \in
B^{\times},$ norms of elements $\frac{\s(1/\bt)}{1/\bt}-1$ and
$\frac{\s(\bt)}{\bt}-1$ generate the same ideal of $A$.

\item \textbf{Trace and $(\s-1)$:} 
 
We have the formal expression $(\s-1)^{p-1}=
\frac{(\s-1)^p}{\s-1}= \frac{\s^p-1}{\s-1}= \s^{p-1}+
\s^{p-2}+...+\s+1$.

Thus, for any $x \in L$, $(\s-1)^{p-1}(x)= \Tr_{L|K}(x)$.
\item \textbf{Reduction:} If we can find an element $x_{\bt}=x
\in L \backslash K$ satisfying an Artin-Schreier equation over $K$ and such that $v(\frac{(\s-1)(x)}{x}) \leq
v(\frac{\s(\bt)}{\bt}-1)$, then we have:

$0 \leq v(N(\frac{(\s-1)(x)}{x}))=t_1 \leq
v(N(\frac{\s(\bt)}{\bt}-1))=t_2$\\\\
After this, it is sufficient to show that the ideal of $A$
generated by $N(\frac{\s(x)}{x}-1)$ is a subset of $H$.
\item \textbf{ $\s-1$ and changes in valuation:} Let $b \in
L^{\times}$.
\begin{itemize}
\item $ \frac{\s(b)-b}{b} \in B \Ra v(\s(b)-b)=v(b)+s_b$ for
some $s_b \geq 0.$
\item For $1 \leq i \leq p-1$, 

$v(\s^i(b)-b) = v(\sum_{1 \leq j \leq i} \s^j(b)-\s^{j-1}(b))
\geq \min_{1 \leq j \leq i} \{v(\s^j(b)-\s^{j-1}(b))
\}=v(\s(b)-b) $

\item By the same argument, applied to $\tau=\s^i$, ($\s=
\tau^m$ for some $1 \leq m \leq p-1$) we have

$v(\s(b)-b) \geq v(\s^i(b)-b)$ and thus, the following equality:

For all $1 \leq i \leq p-1$,

$\dis  v(\s^i(b)-b)=v(b)+s_b$.

\end{itemize}
\end{enumerate}
\begin{proof}
For given $\bt $ as above, we will now construct the special element $x_{\bt}$ [see (O3)] and prove that the ideal of $A$
generated by $N(\frac{\s(\bt)}{\bt}-1)$ is indeed a subset of $H$ . Let $g(T)=\min_K(\bt)$ and $x_{\bt}=x:=
(\s-1)^{p-2}(\frac{\bt^{p-1}}{g'(\bt)})$.

Put $y=\s(x)-x=(\s-1)(x)$. By (O2) and \Cref{trace},
$y=\Tr_{L|K}(\frac{\bt^{p-1}}{g'(\bt)})=1$. As in the case $p=2$,   $y \neq 0$ and we have
$\s(\frac{x}{y})-\frac{x}{y}=  (\s-1)(x)=1$.

Observe that $x=(\s-1)^{p-2}\lb \frac{\bt^{p-1}}{g'(\bt)}\rb \in L
\backslash K$ satisfies $\s(x)=x+1$ and hence, the Artin
Schreier equation $\ap^p-\ap=N(x)$. Thus, we have

\begin{equation}
\frac{1}{N(x)} =N \lb
1/(\s-1)^{p-2}(\frac{\bt^{p-1}}{g'(\bt)})\rb \in H.
\end{equation}
Now we need to relate the principal ideals generated by
$N(\frac{(\s-1)(x)}{x})$ and $N(\frac{\s(\bt)}{\bt}-1)$. For
this, we look at the $L$-valuation of these elements.
Let $v(\frac{(\s-1)(x)}{x})=s' \geq 0$ and
$v(\frac{(\s-1)(\bt)}{\bt})=s \geq 0$.

If $s' \leq s$, then $N(\frac{\s(\bt)}{\bt}-1) \in H$ and hence,
$\lb N(\frac{\s(\bt)}{\bt}-1) \rb =N(\frac{\s(\bt)}{\bt}-1)A
\subset H$.

Now suppose that $s'>s$. Put $r= \frac{\bt^{p-1}}{g'(\bt)}$. Then 
$g'(\bt)=\prod_{1 \leq i \leq p-1} (\bt- \s^i(\bt))$. Hence, by
(O4),
\begin{equation}
 v(r) = -(p-1)s 
\end{equation}
For $1 \leq i \leq p-1$, let
$v((\s-1)^i(r))=v((\s-1)^{i-1}(r))+c_i$; $c_i \geq 0$.
$c_{p-1}=s'$ by definition. Since $v((\s-1)^{p-1}(r))=v(1)=0$,
from (6.3), we see that
\begin{equation}
\sum_{i=1}^{p-1} c_i = -v(r) = (p-1)s
\end{equation}
Let $c:= \inf \{v(\frac{\s(b)}{b}-1)\mid b \in L^{\times}
\}=\inf \{s_b \mid b \in L^{\times} \} \in \R,$ where $s_b$ is
as described in (O4).

\noindent We observe 
$(p-1)s= \sum_{i=1}^{p-2} c_i +s' \geq (p-2)c +s' > (p-2)c +s
\geq (p-1)c \geq 0$.

\noindent In particular, $(p-2)(s-c) \geq s'-s >0$.
By the definition of $c$, we can take $s$ very close to $c$ such
that $s' \leq s$ for this new $s$.

This concludes the  proof.
\end{proof}

\begin{cor}\label{best defect}
Under the assumptions of \Cref{hn}, the following statements are
equivalent:
\begin{enumerate}
\item Best $f$ exists.
\item $H$ is a principal ideal of $A$.
\item $\js$ is a principal ideal of $B$.
\item $L|K$ is defectless.
\end{enumerate}
\end{cor}

\section{Filtered Union in the Defect Case}
 To generalize the results to the
defect case, we write the ring $B$ as a filtered union of rings
$A[x]$, where the elements $x$ are chosen very carefully.
Although these are not valuation rings, each ring is generated
by a single element (over $A$). This makes the extensions
$K(x)|K$ and the corresponding differential modules, special
ideals easier to understand. We will use \Cref{hn} to prove these results.

\begin{thm}\label{fil} Consider $\dis \mathscr{S}= \{ \ap \in L
\mid \ap^p-\ap=f; f \in K $, and $\ap$ generates $L|K \}$. For
each $\ap \in \mathscr{S}$, we can find $\dis \ap' \in
B^{\times} \cap \ap K^{\times} $ such that $B= \cup_{\ap \in
\mathscr{S}} A[\ap']$ is a filtered union, that is, the
following are true:
\begin{enumerate}[(i)]
\item For any $\ap_1, \ap_2 \in \mathscr{S}$, either $A[\ap_1']
\subset A[\ap_2']$ or $A[\ap_2'] \subset A[\ap_1']$.
\item Given any $\bt \in B$, there exists $\ap \in \mathscr{S}$
such that $\bt \in A[\ap']$.
\end{enumerate}
\end{thm}

\subsection{ $p=2$} 
First we consider the  filtered union in the
$p=2$ case, as given by the result below.

\begin{pr} For $p=2$, $B= \cup_{\ap \in L \backslash K}
A[\frac{\Tr(\ap)}{\ap}]$ is a filtered union.
\end{pr}
\begin{proof} We are dealing with the defect case, so
$v_L=v_K=v$. Let $\ap_1, \ap_2 \in L \backslash K$.
$\frac{\Tr(\ap_i)}{\ap_i}= \beta_i \in B$. $\s
(\frac{\ap_i}{\Tr(\ap_i)})=\frac{\s(\ap_i)}{\Tr(\ap_i)}=\frac{\ap_i}{\Tr(\ap_i)}+1$
since $p=2$. We have
$(\frac{\ap_i}{\Tr(\ap_i)})^2-\frac{\ap_i}{\Tr(\ap_i)}=
\frac{1}{c_i}; c_i \in A$.

\noindent $\s (\frac{\ap_1}{\Tr(\ap_1)} -
\frac{\ap_2}{\Tr(\ap_2)})=\frac{\ap_1}{\Tr(\ap_1)} +1 -
\frac{\ap_2}{\Tr(\ap_2)}-1=\frac{\ap_1}{\Tr(\ap_1)} -
\frac{\ap_2}{\Tr(\ap_2)}$\\ Therefore, $\frac{\ap_1}{\Tr(\ap_1)} -
\frac{\ap_2}{\Tr(\ap_2)}=\frac{1}{\bt_1}-\frac{1}{\bt_2}=g \in K$
and $\frac{1}{c_1}=\frac{1}{c_2}+g^2-g$. We note that
$(\frac{1}{\bt_i})^2-\frac{1}{\bt_i}=\frac{1}{c_i}$, that is,
$\bt_i=\frac{c_i}{\bt_i}-c_i.$ We will prove the following two statements. 
\begin{enumerate}[(1)] \item If $v(c_1)> v(c_2)$, then $ A[\bt_1]$
is a subset of $ A[\bt_2].$ \item If $v(c_1)= v(c_2)$, then $ A[\bt_1] =
A[\bt_2]$\end{enumerate}
 In (1), it is enough to show that
$\bt_1 \in A[\bt_2]$. Since $\frac{c_2}{\bt_2}=\bt_2 +c_2,$ it is an element of $A[\bt_2]$. Consequently, $\frac{c_1}{\bt_2}=\frac{c_2}{\bt_2} \frac{c_1}{c_2} \in A[\bt_2]$.

\noindent \textbf{Claim:}
$ \dis \bt_1 \in A[\bt_2] \Lr \frac{c_1}{\bt_1} \in
A[\bt_2] \Lr \frac{c_1}{\bt_2} \in A[\bt_2]$

\noindent \textbf{Proof of Claim:} This can be shown by following steps: \begin{itemize}
\item $\bt_1=\frac{c_1}{\bt_1}-c_1$, $c_1 \in A$.
\item
$\frac{c_1}{\bt_1}=c_1(\frac{1}{\bt_2}+g)=\frac{c_1}{\bt_2}+c_1g$.\\
Now
$-v(c_1)=v(\frac{1}{c_1})=v(\frac{1}{c_2}+g^2-g)<-v(c_2)=v(\frac{1}{c_2})
\leq 0$\\ $ \Ra -v(c_1)=v(\frac{1}{c_2}+g^2-g)=v(g^2-g)=2v(g)$
(the last equality follows from $0>-v(c_1)$.)\\ $\Ra
v(c_1g)=-2v(g)+v(g)=-v(g) >0$\\ $\Ra c_1g \in A$ (since $c_1g$
is already in $K$.)
\end{itemize}

 The proof of (2) is very similar to
the proof of (1). We just need to show that $v(c_1g) \geq 0$. If
$v(g) \geq 0,$ this is clearly true. Let $v(g)<0,
v(c_1)=v(c_2)=v \geq 0$. Since 
$v(\frac{1}{c_2}+g^2-g)=v(\frac{1}{c_2}), 2v(g)= v(g^2-g)
\geq v(\frac{1}{c_2})=-v. $ Hence, $v(c_1g) \geq v-
\frac{v}{2}=\frac{v}{2} \geq 0$.
\end{proof}

\begin{rem} This particular construction in the case $p=2$
doesn't appear to have an easy generalization to the case $p>2$.
We use a different approach.
\end{rem}

\subsection{Some Elementary Results for $p>2$} 

Due to the defect, given any $\ap \in \mathscr{S}$ there exists
$ \g_\ap=\g \in A$ such that $v(\g)=-v(\ap)=-\frac{1}{p}v(f)$.
Define $ \ap'=\ap \g \in B^{\times}$. We claim that this choice
of $\ap'$ satisfies the conditions of \Cref{fil}. We note that
the ring $A[\ap']$ does not depend on the choice of $\g$.

\begin{lm} If $\ap_1, \ap_2 \in \mathscr{S}$ such that $
v(\ap_1) \leq v(\ap_2)$, then $A[\ap_1'] \subset A[\ap_2']$.
\end{lm}
\begin{proof} We have (by choosing appropriate conjugates )
$\s(\ap_2-\ap_1)=(\ap_2+1)-(\ap_1+1)=\ap_2-\ap_1$ . Hence,
$\ap_2-\ap_1=:h \in K$.

$ v(\ap_1) \leq v(\ap_2) \Ra v(\g_1) \geq v(\g_2) $ and $v(h)
\geq v(\ap_1)=-v(\g_1)$. Therefore, $ \frac{\g_1}{\g_2}, \g_1h
\in A$.

Consequently,
$\ap_1'=\g_1(\ap_2-h)=\frac{\g_1}{\g_2}\ap_2'-\g_1h \in
A[\ap_2']$.
\end{proof}

\begin{lm}\label{btap'} Given any $\bt \in B$, there exists $\ap \in
\mathscr{S}$ such that $\lb \s(\bt)-\bt \rb \subset \lb
\s(\ap')-\ap'\rb $.
\end{lm}
\begin{proof} Let $v:=v(\s(\bt)-\bt)$ , $\dis v_0:=\inf_{b \in
B^{\times}} v(\s(b)-b) \in \R$. Hence, $\is=\js=\{x \in
L^{\times} \mid v(x) > v_0\}$ and $\ns=\{x \in K^{\times} \mid
v(x) > pv_0\}$. Since this is the defect case, by \Cref{j},
$\is$ is not a principal ideal. We need to show that $v>c$,
where $c \in \R$ is defined by

\noindent $\dis c:=\inf_{\ap \in
\mathscr{S}}v(\s(\ap')-\ap')=\inf_{\ap \in
\mathscr{S}}v(\g_\ap)=\inf_{\ap \in \mathscr{S}}-v(\ap)=\inf_{f
\in \mathfrak{A}}-\frac{1}{p}v(f)$.

\noindent Note that $\dis H=\{x \in K^{\times} \mid v(x) >
pc\}$. By \Cref{hn}, $H=\ns$ and hence, $c=v_0$. To conclude the
proof, we observe that $\s(\bt)-\bt \in \is \Ra v >v_0=c$.

\end{proof}
\begin{lm}\label{bin} For $x,y \in L$, we have $(\s-1)^n(xy)=
\sum_{k=0}^n {n \choose k}(\s-1)^{n-k}(x)(\s-1)^k(\s^{n-k}(y))
$\\
In particular, for $n=1$,
$(\s-1)(xy)=(\s-1)(x)\s(y)+x(\s-1)(y)$.
\end{lm}
\begin{proof} This can be proved by using induction on $n$ and
the binomial identity ${n \choose k}+{n \choose {k-1}}={n+1
\choose k}$.
\end{proof}

\subsection{Filtered Union for $ p>2$}

\begin{pr}\label{btap} Given any $\bt \in B^{\times}$, there
exists $\ap \in \mathscr{S}$ such that
$(\s-1)^{p-1}(\frac{1}{F'(\ap')} A[\ap', \bt]) \subset B$. Here,
$F$ denotes the minimal polynomial of $\ap'$ over $K$.
\end{pr}
\begin{proof} 
We compute valuation of these elements and show that it is
non-negative.

For all $\ap \in \mathscr{S}, 1 \leq k \leq p-1 $,
$\s^k(\ap')-\ap'=(\s^k(\ap)-\ap)\g=k\g$. Therefore,
$F'(\ap')= - \g^{p-1}$. In particular, it is an element of
$K$ and hence, fixed by $\s$.

We wish to select $\ap$ such that for all $i,j \geq 0$,

\begin{equation} v((\s-1)^{p-1}(\ap'^i \bt^j)) \geq
v(F'(\ap'))=(p-1)v(\g)
\end{equation}

\begin{enumerate}[(Step 1)]
\item \textbf{Construction of the special $\ap'$}\\We begin with
an $\ap_0$ satisfying $\lb \s(\bt)-\bt \rb \subset \lb
\s(\ap_0')-\ap_0' \rb $. Let $(\s-1)(\bt)=b_1\g_0; b_1 \in B$.
Therefore, $(\s-1)^2(\bt)=(\s-1)(b_1)\g_0$. We don't know much
about the valuation of $(\s-1)(b_1)$, however. Let $\ap_1$ be
such that $\lb (\s-1)(b_1) \rb \subset \lb (\s-1)(\ap_1') \rb$.
Write $(\s-1)(b_1)=b_2 \g_1; b_2 \in B$. Now we can write
$(\s-1)^2(\bt)= b_2 \g_1\g_0$. Using this process, we can find
$b_i$'s and $\ap_i$'s such that $(\s-1)^i(\bt)= b_i
\g_{i-1}...\g_1\g_0;$ where $b_i \in B$.

Let $\g$ be the $\g_j$ with smallest valuation involved in the
expression for $i=p-1$. Let $\ap$ denote the corresponding
$\ap_j$. We will show that this $\ap$ satisfies the required
property (5.8).

\item \textbf{Proof for $\bt$}\\$\lb \s(\bt)-\bt \rb \subset \lb
\s(\ap_0')-\ap_0' \rb \subset \lb \s(\ap')-\ap' \rb=\lb \g \rb$,
since $v(\g) \leq v(\g_0)$.
Due to the choice of $\g$, we also have $v((\s-1)^t(\bt)) \geq t
v(\g)$ for all $1 \leq t \leq p-1$. In particular, this is true
for $t=p-1$, proving the statement (5.8) for the case $i=0,
j=1$.

\item \textbf {Terms $\ap'^i \bt^j$}\\ For the terms of the form
$\bt^j$, we use induction on $j$ and \Cref{bin}. Valuation of
each term in the expansion is at least $(p-1)v(\g)$. In fact, by
a similar argument, $v((\s-1)^k(\bt^j)) \geq k v(\g)$ for all $1
\leq k \leq p-1$. \\ For the general terms $\ap'^i \bt^j$, first
note that $(\s-1)^k(\ap')=(\s-1)^{k-1}(\g)=0$ for all $k >1$.
Therefore, (again using the identity), we have\\
$(\s-1)^{p-1}(\ap'^i \bt^j)=\ap'^i(\s-1)^{p-1}(\bt^j)+
(p-1)(\s-1)(\ap'^i)(\s-1)^{p-2}(\s(\bt^j))$. Once again, both
these terms have valuation $\geq(p-1)v(\g)$.
\end{enumerate}
This concludes the proof of the proposition.
\end{proof}

\subsection{Proof of \Cref{fil}} Let $\bt$ and corresponding
special $\ap'$ be as described above in (Step 1). We recall that
for an $A$-module $R \subset L$, $R^*:= \{ x \in L \mid
\Tr_{L|K}(xR) \subset A\}$.
\begin{enumerate}
\item $A[\ap', \bt]^*=A[\ap']^*$
\begin{proof}
Clearly, $A[\ap', \bt]^* \subset A[\ap']^*= \frac{1}{F'(\ap')}
A[\ap']$. We proved that $(\s-1)^{p-1}(\frac{1}{F'(\ap')}
A[\ap', \bt]) \subset B$. Since $(\s-1)^{p-1}=\Tr_{L|K}$ has
image in $K$, $\Tr_{L|K}(\frac{1}{F'(\ap')} A[\ap', \bt])
\subset B \cap K =A$ and we have the reverse inclusion.
\end{proof}
\item $R:=A[\ap', \bt], S:=A[\ap']$ are finitely generated free
$A$-modules.
\begin{proof} Since $\bt , \ap'$ are integral over $A$, $R$ and
$S$ are finitely generated $A$-modules. $A$ is a valuation ring
and $R,S$ are finitely generated torsion-free $A$-modules.
Therefore, $R, S$ are free $A$-modules (of finite ranks).
\end{proof}
\item $A[\ap', \bt]=A[\ap']$
\begin{proof} $R$ is a free $A$-module of finite rank. Hence, $
R^{**}=(R^*)^*=R$. Similarly, $S^{**}=S$. By (1), $R^*=S^*$ and
hence, $R=S$.
\end{proof}
\end{enumerate}

These statements, in combination with \Cref{btap} prove part
(ii) of \Cref{fil}. Part (i) was already proved in Lemma 5.4.
This concludes the proof.

\section{Proof of \Cref{comm dia}}

\begin{lm}\label{norm} $N_{L|K}=N: B \to A/(\is \cap A)$ is a
surjective ring homomorphism.
\end{lm}
\begin{proof} We just need to check the additive property of $N:
B \to A/(\is \cap A)$ in order to prove that it is a ring
homomorphism.
For $x \in B, N(x)= x \prod_{i=1}^{p-1} \s^i(x)$.

For each $1 \leq i \leq p-1, \s^i(x) \equiv x\mod \is.$

Thus, $N: B \to B/\is$ is just the $p$-power map, that is, $x
\mapsto x^p \mod \is$ and hence, additive. This makes $N: B \to
A/(\is \cap A)$ additive as well.
\end{proof}

\begin{rem} We don't need any assumptions regarding defect or
rank here.
\end{rem}

\subsection{Case I: Relation between the ideals $H, \I, \is,
\js$}

\begin{notn}
Case I is the defectless case, so best $f$ exists and we can
define the ideal $\I$ of $A$ by

$\dis \I:= \lb \{ \frac{a}{f} \in K \mid v_K(f+a^p-a)=v_K(f) \}
\rb$. It is worth noting that this definition coincides with the
one in Lemma 2.9. Let $v_L(\ap)=-v_0 \leq 0$. Hence,
$v_L(f)=-pv_0$, $H=\{x \in K \mid v_L(x)\geq pv_0 \}$, $\I =\{x
\in K \mid v_L(x)\geq (p-1) v_0 \}$ and $\js=\{x \in L \mid
v_L(x)\geq v_0 \}$
\end{notn}
\begin{pr}\label{rsw well} $ H \subset \I \subset \is \cap A$.
\end{pr}
\begin{proof} 
Comparing valuations mentioned above, it is clear that $H
\subset \I$.

We break down the rest of the argument into several cases:
\begin{itemize}
\item If $e_{L|K}=1, \is=\js=\{x \in L \mid v_L(x)\geq v_0 \}$
and the result follows.
\item Let $e_{L|K}=p$.
\begin{enumerate}[(i)]
\item $p>2$

$\frac{1}{\ap} \in B \Ra \s \lb \frac{1}{\ap} \rb -
\frac{1}{\ap}= \frac{-1}{\ap (\ap+1)} \in \is$.

Hence, $\{x \in L \mid v_L(x)\geq 2 v_0 \} \subset \is $. Since
$p>2, \ p-1 \geq 2$ and hence, $\I \subset \is \cap A$. We
cannot use this argument for $p=2$, since in that case,
$p-1=1<2$.
\item $p=2$  

Let $\frac{a}{f} \in K$ such that $ v_K(f+a^2-a)=v_K(f)$.
Consider $ b= \frac{a \ap}{f}.$ Then

\noindent $ v_L(b)=v_L(\ap)+v_L(\frac{a}{f}) \geq -v_0+v_0=0 \Ra
b \in B \Ra \s(b)-b \in \is.$

$\s(b)-b=\s(b)+b=\Tr(b)=\frac{a}{f} \\ \Tr(\ap)=\frac{a}{f}$
since $\\ \Tr(\ap)=1$.

\noindent Hence, $\frac{a}{f} \in \is \cap K = \is \cap A$. This
concludes the proof.
\end{enumerate}

\end{itemize}
\end{proof}

\subsection{Case I} Let $f$ be best, $b \in B$. We prove that
the following diagram commutes:

\noindent $
\begin{tikzcd}[column sep=large]
\ol/\js \ol \arrow{r}{\varphi_\s}[swap]{\cong}
\arrow{d}[swap]{\dn}
&\js/\js^2  \arrow[hook]{d}{\overline{N}}\\
\ok/(\is \cap A )\ok   &H/H^2 \arrow{l}{\rsw}
\end{tikzcd}$ where the maps are given by $
\begin{tikzcd}[column sep=large]
b \dl \ap \arrow{r}{\varphi_\s}[swap]{} \arrow{d}[swap]{\dn}
&b \frac{1}{\ap}  \arrow{d}{\overline{N}}\\
N(b) \dl f  &N(b) \frac{1}{f} \arrow{l}{\rsw}
\end{tikzcd}$

\begin{proof}
Consider the map $\varphi_\s: \ol/ \js \ol \to \js/\js^2$. By
\Cref{phi}, we know that $\varphi_{\s}$ is a surjective
$B$-module homomorphism. We prove that it is injective.

Since $\ol$ is generated by $\dl \ap$, it is enough to
consider elements of the form $b \dl \ap ;$ where $ b \in B$.
$b \dl \ap \in Ker (\varphi_\s) \Lr b
(\frac{\s(\ap)}{\ap}-1)=b \frac{1}{\ap} \in \js^2 \Lr b \in \js
\Lr b \dl \ap \in \js \ol$. Therefore, $\varphi$ is a
$B$-module isomorphism.

Next, we note that $H$ is generated by $\frac{1}{f}$ and
$N(\ap)=f$. By \Cref{norm}, we have additivity of the two
vertical maps. Since $ H \subset \I \subset \is \cap A$, the map
$\rsw$ is independent of the choice of best $f$.

\end{proof}

\subsection{Preparation for  Case II}

\subsubsection{Valuation on $A$ and $B$: } Fix some $\ap_0 \in
\mathscr{S}$ as our starting point. We may only consider $\ap
\in \mathscr{S}$ such that $v(\ap_0) <v(\ap)$. Consider the
subset $\mathscr{S}_0$ of $\mathscr{S}$ consisting of such
$\ap$'s. Let $v(\ap_0)=-\mu <0, \g_0 \in A$ such that $
v(\g_0)=\mu$. For each $\ap \in \mathscr{S}_0$, we have
corresponding $\g_\ap \in A$ with $v(\g_\ap)=-v(\ap)
<v(\g_0)=\mu$ and $\ap'=\ap \g_\ap \in B^{\times}$. Let $F_\ap$
denote the minimal polynomial of $\ap'$ over $K$. We recall that $F'_\ap(\ap')=-\g_\ap^{p-1}
$ and hence, we have the isomorphism $\Om^1_{A[\ap']|A} \cong A[\ap']/(\g_\ap^{p-1})$ described in 6.3.3 .\\ Let $f_\ap:=\ap^p-\ap=N(\ap) \in K$.
 
\subsubsection{Special Ideals} Due to the defect, we have
$\is=\js$ by \Cref{ij}.\\ Let $v_0:= \inf \{v(\frac{\s(b)}{b}-1) \mid b \in
B^{\times} \} \in \R$. Then
 
\begin{enumerate}[(a)]
\item $\is=\js=\{b \in B \mid v(b)>v_0\}$, and consequently, by
\Cref{hn} ,
\item $\ns= \{a \in A \mid v(a)> pv_0\}=H$.
\end{enumerate}

We have $\inf \{v(\frac{\s(b)}{b}-1) \mid b \in B^{\times}
\}=\inf \{v(\s(b)-b) \mid b \in B^{\times} \}=\inf
\{v(\s(\ap')-\ap') \mid \ap \in \mathscr{S}_0 \} \in \R$. The last equality follows from \Cref{btap'}.
Therefore,
\begin{equation}
v_0= \inf \{v(\g_{\ap}) \mid \ap \in \mathscr{S}_0  \} \in \R
\end{equation}

\subsubsection{Differential Modules $\Om^1_{A[\ap']|A}$'s } 

We compare $\Om^1_{A[\ap_0']|A}$ and $\Om^1_{A[\ap']|A}$. Let
$c_\ap:=\g_\ap^{p-1}$, $c_0:=\g_0^{p-1}$ and the ratio
$\g_0/\g_\ap =: a_\ap \in A$. Then we have the following
commutative diagram:

\begin{center} $\dis
\begin{tikzcd}[column sep=large]
\Om^1_{A[\ap_0']|A} \arrow{r}{\cong} \arrow[hook]{d}{\rho_\ap}
&A[\ap_0']/(c_0) \arrow{r}{\cong} \arrow[hook]{d}{\iota_\ap}
&(\frac{1}{a_0})A[\ap_0']/(\frac{c_0}{a_0})A[\ap_0']
\arrow[hook]{d}{j_\ap}\\
\Om^1_{A[\ap']|A} \arrow{r}{\cong} &A[\ap']/(c_\ap)
\arrow{r}{\cong}
&(\frac{1}{a_\ap})A[\ap']/(\frac{c_\ap}{a_\ap})A[\ap']
\end{tikzcd}$ 
\end{center}
Here, $a_0=\g_0/\g_0=1 \in A$ and the isomorphisms are given by
$\dis
b_0 d \ap_0' \ \mapsto
b_0 \ \mapsto \frac{b_0}{a_0}; \ $ for all $ \ b_0 \in
A[\ap_0']$ and $
b d \ap' \  \mapsto
 b \ \mapsto \frac{b}{a_\ap}; \ $ for all $ \ b \in A[\ap']
$. The vertical maps are described as follows. We look at the
relationship between the generators $\ap_0', \ap'$ and
similarly, between $d \ap_0' \ , \ d \ap'$. Since $\ap$ and
$\ap_0$ give rise to the same extension $L|K$, $\ap_0-\ap=: h
\in K$. Comparing the valuations, we see that $v(\ap_0)=v(h) <
v(\ap)$ and hence, $u=h \g_0 \in A^{\times}$.

\begin{equation} \ap_0'=(\ap+h) \g_0= (\ap+h) \g_\ap \cdot \ a_\ap=
a_\ap \ap' + u
\end{equation}
Since $\ap' \in B$ and $a_\ap , u \in A$, $\dis \ap' d a_\ap =0=
du$ in the differential module $\dis \Om^1_{A[\ap']|A}$.
Therefore, we have
\begin{equation}
d \ap_0'=a_\ap d \ap' + \ap' d a_\ap + du = a_\ap d \ap' 
\end{equation}
Thus, $\rho_\ap$, $\iota_\ap$ are given by multiplication by
$a_\ap$. The map $j_\ap$ is also multiplication by $a_\ap$ and
rises from the inclusions

\begin{equation} (\frac{1}{a_0})A[\ap_0'] \subset
(\frac{1}{a_\ap})A[\ap']; \ \frac{1}{a_0} \mapsto \frac{1}{a_0}
a_\ap=\frac{1}{a_\ap} a_\ap^2 \end{equation} and
\begin{equation} (\frac{c_0}{a_0})A[\ap_0'] \subset
(\frac{c_\ap}{a_\ap})A[\ap']; \ \frac{c_0}{a_0} \mapsto
\frac{c_0}{a_0} a_\ap=\frac{c_\ap}{a_\ap} a_\ap^p \end{equation}

\begin{lm}\label{diff} Consider the fractional ideals $\T$ and
$\T'$ of $L$ given by $\T= \{x \in L \mid v(x) >v_0-\mu \}$ and
$\T'=\{ x \in L \mid v(x) > pv_0-\mu \}$. Then we have:
\begin{enumerate}[(a)]
\item $\Om^1_{B|A} \cong \T/\T'$ 
\item $\T/ \js \T \cong \js/\js^2$ 
\end{enumerate}
\end{lm}
\begin{proof}
\begin{enumerate}[(a)]
\item Let $I$ be the fractional ideal of $L$ generated by the
elements $(\frac{1}{a_\ap})$. Let $I'$ be the fractional ideal
of $L$ generated by the elements $(\frac{c_\ap}{a_\ap})$. Under
the isomorphisms described in the preceding discussion, we can
identify each $\dis \Om^1_{A[\ap']|A}$ with $
(\frac{1}{a_\ap})A[\ap']/(\frac{c_\ap}{a_\ap})A[\ap']$. Taking
limit over $\ap$'s, we can identify $\Om^1_{B|A}$ with $I/I'$.

\noindent Since $-v(a_\ap)=v(\g_\ap)-v(\g_0)= v(\g_\ap)-\mu$,
$\dis I=\{x \in L \mid v(x) > \inf_\ap v(\g_\ap)-\mu \} = \T$.
Similarly, $v(c_\ap)=(p-1)v(\g_\ap) \Ra
v(\frac{c_\ap}{a_\ap})=pv(\g_\ap)-\mu \Ra I'=\T'$.

\item This follows from the fact that $\T \cong \js $ as
$B$-modules, via the map $\times \g_0 : x \mapsto x \g_0$.
\end{enumerate}
\end{proof}

\subsection{Proof of \Cref{comm dia} in Case II} 
Due to the defect, we consider $\OL$ and $\OK$ instead:

\begin{center}$
\begin{tikzcd}[column sep=large]
\OL/\js \OL \arrow{r}{\varphi_\s}[swap]{\cong}
\arrow{d}[swap]{\dn}
&\js/\js^2  \arrow[hook]{d}{\overline{N}}\\
\OK /(\is \cap A )\OK   &H/H^2 \arrow{l}{\rsw}
\end{tikzcd}$
\end{center}

As discussed in \Cref{diff}, we can write $\dis \OL=
\varinjlim_{\ap \in \mathscr{S}_0} \Om^1_{A[\ap']|A}$ and it is
enough to consider the diagram for each $\ap \in \mathscr{S}_0$:

\begin{equation}
\begin{tikzcd}[column sep=large]
\Om^1_{A[\ap']|A}/(\frac{1}{\ap})A[\ap']\Om^1_{A[\ap']|A}
\arrow{r}{\varphi_\s}[swap]{\cong} \arrow{d}[swap]{\dn}
&(\frac{1}{\ap})A[\ap']/(\frac{1}{\ap})^2 A[\ap']
\arrow[hook]{d}{\overline{N}}\\
\OK /(\is \cap A )\OK &(\frac{1}{f_\ap})A/(\frac{1}{f_\ap})^2A
\arrow{l}{\rsw}
\end{tikzcd}\end{equation}

 where the maps are given by 

\begin{center}$\begin{tikzcd}[column sep=large]
b d  \ap' \arrow{r}{\varphi_\s}[swap]{} \arrow{d}[swap]{\dn}
&b\ap' \frac{1}{\ap}  \arrow{d}{\overline{N}}\\
N(b\ap') \dl f_\ap  &N(b\ap') \frac{1}{f_\ap} \arrow{l}{\rsw}
\end{tikzcd}$\end{center}

We note that in $\ol, \dl \ap= \dl \ap' + \dl \g_\ap= \dl \ap'=\frac{d \ap'}{\ap'}$ and
$\frac{\s(\ap')}{\ap'}-1=\frac{1}{\ap}$.

At each $\ap$-level, we observe the following:
\begin{enumerate}[(i)]
\item The map $\varphi_\s:
\Om^1_{A[\ap']|A}/(\frac{1}{\ap})\Om^1_{A[\ap']|A} \to
(\frac{1}{\ap})/(\frac{1}{\ap})^2$ is same as the one obtained
from \Cref{diff}.
\begin{proof} By \Cref{diff},
$\Om^1_{A[\ap']|A}/(\frac{1}{\ap})\Om^1_{A[\ap']|A} \cong
(\frac{1}{a_\ap})/(\frac{1}{\ap})(\frac{1} {a_\ap}) \cong
(\frac{1}{\ap})/(\frac{1} {\ap})^2$ under the composition $d
\ap' \mapsto \frac{1}{a_\ap} \mapsto \g_0
\frac{1}{a_\ap}=\g_\ap=\frac{\ap'}{\ap}$.

On the other hand, $\varphi_\s(d \ap')=\ap' \lb
\frac{\s(\ap')}{\ap'}-1 \rb = \frac{\ap'}{\ap}$.
\end{proof}
\item The map $\rsw$ is well-defined. 

\begin{proof} 

Define the ideal $\I_\ap$ of $A$ by $\I_\ap:=\lb \{
\frac{a}{f_\ap} \in K \mid v_K(f_\ap+a^p-a)=v_K(f_\ap) \} \rb$.
As in case (I), we have $ (\frac{1}{f_\ap})A \subset \I_\ap
\subset (\frac{1}{\ap})A[\ap'] \cap A$. Since
$(\frac{1}{\ap})A[\ap'] \cap A \subset \js \cap A = \is \cap A$,
the map $\rsw$ is well-defined.
\end{proof}
\end{enumerate}

\section{The Different Ideal $\D_{B|A}$}
\subsection{Basic Properties} We recall that $\dis
\D_{B|A}^{-1}:= \{x \in L \mid \Tr_{L|K}(xB) \subset A
\}=B^*$ and the different ideal $\D_{B|A}$ is defined to be its
inverse ideal.

\begin{lm} Let $\mu \in B \backslash A, L=K(\mu)$, and $F(T) \in
K[T]$ the minimal polynomial of $\mu$ over $K$, then $\dis
A[\mu]^*= \frac{1}{F'(\mu)}A[\mu]$.
\end{lm}
\begin{proof} See lemma 6.76 of \cite{K}.
\end{proof}

Now we describe the different ideal $\D_{B|A}$ in the cases I
and II. We will assume that the extension $L|K$ is ramified.
Consider the following three sub-cases:
\begin{itemize}
\item \textbf{Case (i):}  $e_{L|K}=1, f_{L|K}=p$.

Let $v$ denote both $v_L$ and $v_K$. Assume that $L|K$ is
generated by $\ap^p-\ap=f$ where $f$ is best. There exists $\g
\in A$ such that $\ap':=\ap \g \in B^{\times}$ and $l|k$ is
purely inseparable, generated by the residue class of $\ap'$.
Let $v(\ap)=-v_0$. Hence, $v(f)=-pv_0, v(\g)=v_0$.

Since $ f\g^p \in A^{\times}, F(T)=T^p-T\g^{p-1}-f\g^p $ is the
minimal polynomial of $\ap'$ over $A$. Therefore, $
F'(T)=pT^{p-1}-(p-1)\g^{p-1}=\g^{p-1}$. By \Cref{cyclic},
\Cref{dmu}, $B=A[\ap']$ and hence, $ \D_{B|A}^{-1}= B^*=
A[\ap']^*= \frac{1}{F'(\ap')}A[\ap']$ is clearly a fractional
ideal of $L$, generated by a single element
$\frac{1}{F'(\ap')}$.

\item \textbf{Case (ii):} $e_{L|K}=p, f_{L|K}=1$.

Let $f$ be best, $v_L(\ap)=-v_0$. Recall that $B=
\sum_{i=0}^{p-1} A_i\ap^i ; A_0:=A$, for all $1 \leq i \leq
p-1,\\ A_i := \{ x \in K \mid v(x) \geq iv_0 \}= \{ x \in A \ |
\ v(x) > iv_0 \}$. Let $y \in L$. Then for all $ 0 \leq i \leq
p-1$,

\begin{equation} y = \sum_{j=0}^{p-1} y_j\ap^j \in
\D^{-1}_{B|A}; y_j \in K \Lr \Tr_{L|K}(y\ap^iA_i) \subset A
\end{equation}

\noindent $\ap$ has the minimal polynomial $F(T)=T^p-T-f$.
Hence, $F'(\ap)=-1$.\\ For $1 \leq i \leq p-1,
\ap^{i+(p-1)}=\ap^i+f \ap^{i-1}$.
By \Cref{trace}, we have

$\dis \Tr_{L|K}(\ap^{i})= \left\{ \begin{array}{c l} 0 ;
\hspace{30pt} 0 \leq i \leq p-2 \\
-1 ; \hspace{30pt}  i = p-1, 2(p-1)  \\
0 ; \hspace{30pt} p \leq i \leq 2(p-1)-1  
\end{array} \right. $

Let $x_i \in A_i$. Then 

$\dis \Tr(x_i y \ap^{i})= \Tr(\sum_{j=0}^{p-1} x_iy_j\ap^{i+j})=
\left\{ \begin{array}{c l} -x_0y_{p-1} ; \hspace{30pt} i=0 \\
-x_iy_{p-1-i} ; \hspace{30pt} 1 \leq i \leq p-2  \\
-x_{p-1}y_0-x_{p-1}y_{p-1} ; \hspace{30pt}  i =p-1  
\end{array} \right. $

Hence, $ y \in \D^{-1}_{B|A}$ if and only if
$A_0y_{p-1},A_{p-1}(y_0+y_{p-1}), A_iy_{p-1-i} \subset A\ ($ for
all $ \ 1 \leq i \leq p-2)$.

\item \textbf{Case (iii):} Rank $1$ and $e_{L|K}=1, f_{L|K}=1$

Let $\Gamma \subset \R$ and let $v$ denote both $v_L, v_K$. By
\Cref{fil}, we can write $B= \cup_{\ap \in \mathscr{S}}
A[\ap']$, where $ \ap'=\ap\g_\ap \in B^{\times}, \ \g_\ap \in
A$. Recall that $\dis v_0:= \inf_{\ap \in \mathscr{S}}v(\g_\ap)
\in \R \backslash \Gamma$. By an argument similar to Case (i)
above, we have $\D^{-1}_{A[\ap']|A}=\{x \in L \mid v(x) \geq
(p-1)v(\ap)=-(p-1)v(\g_\ap) \}$.\\
Since all the $A[\ap']$'s and $B$ have the same fraction field
$L, \D^{-1}_{B|A} \subset \D^{-1}_{A[\ap']|A} \ $ for all $ \
\ap \in \mathscr{S}$. Hence, $\g_\ap^{p-1} \D^{-1}_{B|A} \subset
\g_\ap^{p-1} \D^{-1}_{A[\ap']|A} \subset A[\ap'] \subset B$ and 
$\D^{-1}_{B|A}$ is a fractional ideal of $L$ described by
\begin{align*}
\D^{-1}_{B|A}&= \cap_{\ap \in \mathscr{S}}
\D^{-1}_{A[\ap']|A}\\&=\cap_{\ap \in \mathscr{S}}\{x \in L \mid
v(x) \geq (p-1)v(\ap) \}\\ &=\{x \in L \mid v(x) \geq
(p-1)v(\ap) \ \forall \ \ap \in \mathscr{S} \}\\ &=\{x \in L \ |
\ v(x) \geq -(p-1)v_0 \}
\end{align*}
\end{itemize}

\subsection{Results in the case $e_{L|K}=1$} Let $L|K$ satisfy
(I) or (II) and assume further that $e_{L|K}=1$.

\begin{lm} $\{x \in L \mid \Tr_{L|K}(xB) \subset H \} = \js$.\end{lm}
\begin{proof} Since $e_{L|K}=1,$ given any $x \in L$, there are
elements $ x' \in B^{\times}, a \in K$ such that $x=x'a$. Hence,
$\Tr(xB) = a\Tr(x'B)=a\Tr(B).$
\begin{itemize}
\item \textbf{Case(i):} We note that $\Tr (\frac{1}{\ap})=
\frac{-1}{f}$. Hence, $\js = \lb \frac{1}{\ap} \rb B \subset \{x
\in L \mid \Tr(xB) \subset H \}$. Conversely, suppose that
$\Tr(xB) \subset H=\lb \frac{1}{f} \rb A$. In particular, $a\Tr\lb
\frac{1}{\ap'} \rb = a\Tr \lb \frac{1}{\ap \g} \rb =
\frac{a}{\g}\Tr \lb \frac{1}{\ap} \rb = \frac{a}{\g}\lb
\frac{-1}{f} \rb \in H$. Hence, $\frac{a}{\g} \in A \Ra a \ap
\in B \Ra a \in \js$.

\item \textbf{Case (iii):} The argument is very similar to the
case (i). Again, $\js \subset \{x \in L \mid \Tr(xB) \subset H
\}$. Conversely, suppose that $\Tr(xB) \subset H$. Hence, for all
$ \ap \in \mathscr{S}$ ,

$\frac{a}{\g_\ap}\lb \frac{-1}{f_\ap} \rb \in H \\\Ra
v(a)-v(\g_\ap)-v(f_\ap) > pv_0 \\ \Ra v(a) >
(p-1)(v_0-v(\g_\ap))+v_0$.\\
Since this is true for all $\ap \in \mathscr{S},$ we have $ v(a)
\geq v_0$. \\ But $v_0 \notin \Gamma \Ra v(a) > v_0 \Ra a \in
\js$.
\end{itemize}
\end{proof}
\begin{lm} Consider the rank $1$ case, i.e., case (II). For an
ideal $I$ of $A$ and $a \in K, aI \subset I$ if and only if $a
\in A$.
\end{lm}
\begin{cor} In particular, if $L|K$ satisfies (II) and
$e_{L|K}=1$, then $\{x \in L \mid \Tr(x \js) \subset H \}= B$.
\end{cor}
\begin{proof} By Lemma 8.3, $\{x \in L \mid \Tr(x \js) \subset H
\}= \{x \in L \mid x \js \subset \js \}$ and hence, clearly
contains $B$. The reverse inclusion follows from Lemma 8.4.
\end{proof}

\begin{pr}\label{gooddiff} In the cases (i) and (iii), $\D^{-1}_{B|A}$ is
described by:
\begin{itemize}
\item \textbf{Case (i):} $\D^{-1}_{B|A}= \js^{1-p}$ and 

\item \textbf{Case (iii):} $\D^{-1}_{B|A}=\{x \in L \mid xBH
\subset \js \}$.
 \end{itemize}
\end{pr}
\begin{proof} Since $e_{L|K}=1, \is=\js$.
\begin{itemize}
\item \textbf{Case (i):} 
$ v(F'(\mu))= (p-1)v(\g)=(p-1)v_0 \Ra \D_{B|A}^{-1}=\{x \in L \
| \ v(x) \geq -(p-1)v_0 \}$. The rest follows from $\js= \is
=\lb \frac{1}{\ap}\rb B$.

\item \textbf{Case (iii):} By Lemma 8.4, $\Tr(xB) \subset A$ if
and only if $ \Tr(x B)H \subset H$. By Lemma 8.3, $ \Tr(x B)H
\subset H$ if and only $xBH \subset \js$.
\end{itemize}
\end{proof}

\subsection{Results in the case $e_{L|K}=p$}

We study the case (ii) in this section.
\subsubsection{Preparation}
\begin{lm} Let $S$ be a fractional ideal of $L$ and $\ap \in
L^{\times}$ such that $v_L(\ap)$ generates
$v_L(L^{\times})/v_L(K^{\times})$. Then for $y =
\sum_{j=0}^{p-1} y_j\ap^j ;\ y_j \in K, y \in S$ if and only if
$ y_i\ap^i \in S \ $ for all $ \ 0 \leq i \leq p-1$.
\end{lm}
\begin{proof} Since $e_{L|K}=p, v_L(y_i\ap^i); y_i \neq 0$ are
all distinct. If $y \in S$, then for some $s \in S$, we have
$\dis v_L(y)=\min_{ 0 \leq i \leq p-1}v_L(y_i\ap^i) \geq
v_L(s)$. Thus, $ v_L(y_i\ap^i) \geq v_L(s) \ $ for all $ \ 0
\leq i \leq p-1$ and hence, $ y_i\ap^i \in S \ $ for all $ \ 0
\leq i \leq p-1 $. The converse is clearly true.
\end{proof}

Two important applications of the lemma are below.
\begin{itemize}
\item Consider $S=\D^{-1}_{B|A}, y \in L$. $y \in \D^{-1}_{B|A}
\Lr \Tr(y_i \ap^ib) \in A \ $ for all $ b \in B \ $ for all $ \ 0
\leq i \leq p-1$. \\
Hence, $\dis \D^{-1}_{B|A}= \cup_{0 \leq i \leq p-1} \D_i B$
where $\D_i:=\{y\ap^i \mid y \in K, y\ap^i \in \D^{-1}_{B|A}
\}$.

\noindent Fix some $i$, let $y \in K$. Write $\dis b=
\sum_{j=0}^{p-1} x_j \ap^j \ ; \ x_j \in A_j$. $\dis Tr(y\ap^ib)
\in A \Lr \sum_{j=0}^{p-1} yx_j \Tr(\ap^{i+j}) \in A $.

\noindent Thus, if $i=p-1,$ then\\
$ y\ap^{p-1} \in \D^{-1}_{B|A} \Lr v_L(y)+v_L(x_0+x_{p-1}) \geq
0 \ $ for all $ x_0 \in A, \ $ for all $ x_{p-1} \in A_{p-1}. \\
\Lr v_L(y) \geq 0$ and hence,
$\D_{p-1}B=A\ap^{p-1}B=\ap^{p-1}B=\js^{-(p-1)}$.

\noindent If $0 \leq i \leq p-2, \\ y\ap^i \in \D^{-1}_{B|A} \Lr
v_L(y)+v_L(x_{p-1-i}) \geq 0 \ $ for all $ x_{p-1-i} \in
A_{p-1-i} \\ \Lr y\ap^i.x_{p-1-i}\ap^{p-1-i} \in \ap^{p-1}B\\
\Lr y \ap^i A_{p-1-i}\ap^{p-1-i} \subset \ap^{p-1}B$.

\item Consider $S= \is$.

\noindent $\is$ is generated by $\{(\s-1)(x_i \ap^i) \mid x_i
\in A_i, 1 \leq i \leq p-1 \}$. For a fixed $i$, \\
$ (\s-1)(A_i \ap^i)B=A_i \ap^i [(1+\frac{1}{\ap})^i-1]B = A_i
\ap^i \frac{1}{\ap}B=A_i\ap^i \js$.
Thus, $\dis \is = [\cup_{1 \leq i \leq p-1} A_i \ap^i B] \js. $

\end{itemize}
\begin{defn} We consider the $B$-sub-module ${\OL}'$ of $\OL$
generated by the set $\{db \mid b \in \m L \}$ of generators
(and the relations described for $\OL$).

\end{defn}
\begin{lm}\label{ol'ol} ${\OL}' \cong \OL$ as $B$-modules.
\end{lm}
\begin{proof} ${\OL}' \to \OL$ is the map $db \mapsto db$.
Consider the map $\pi: \OL \to {\OL}'$ described below.

For $b \in B$,there exists $x \in A$ such that $b-x \in \m L$.
We define $\pi(db)=d(b-x)$. Note that this definition is
independent of the choice of $x$. It is enough to show that
$\pi$ preserves the relations.

Let $b,c \in B, x, y \in A$ such that $b-x, c-y \in \m K$. 

\noindent Additivity is preserved, since $\pi(d(b+c))=d
(b+c-x-y)=d(b-x)+d(c-y)=\pi(db)+\pi(dc)$.

\noindent Since $dx=0, dy=0$ and $bc-xy=c(b-x)+x(c-y) \in \m L$,\begin{align*}
cd(b-x)+bd(c-y)
&=cd(b-x)+(b-x)dc-(b-x)dc+(b-x)d(c-y)+xd(c-y)+(c-y)dx
\\&=d(c(b-x))+d(x(c-y))+(b-x)d(c-y)-(b-x)dc
\\&=d(bc-xc+xc-xy)+(b-x)[d(c-y)-dc]\\&=d(bc-xy)-(b-x)dy=d(bc-xy)
\end{align*}
Hence, $\pi(d(bc))=c\pi(db)+b\pi(dc)$.
\end{proof}
 
We do not have a good description, as in \Cref{gooddiff}, of the different ideal in this case. However, with further assumptions on the value group $\Gamma_K$, we obtain similar results.

\subsubsection{Some Results in a Special Case}

\begin{notn}\label{notn} Let $L|K$ satisfy (II). Assume further that
$e_{L|K}=p$ and the value group $\Gamma_K$ of $K$ (as an ordered
subgroup of $\R$) is not isomorphic to $\Z$. Thus, $L|K$ is a
defectless Artin-Schreier extension and $\Gamma_K$ is a dense
ordered subgroup of $\R$.
\end{notn}

\begin{lm}\label{ideals}Under the assumptions above (\Cref{notn}),

\begin{enumerate}[(a)]
\item For $1 \leq i \leq p-1, A_iB=\js^i \m L$.
\item $\is = \js \m L.$
\item $\m L^n=\m L$ for all integers $n \geq 1$, and consequently, $\is^n = \js^n \m L.$
\end{enumerate}
\end{lm}
\begin{proof} \begin{enumerate}[(a)]
\item For $1 \leq i \leq p-1, A_iB=\{x \in K \mid v_L(x)
>iv_0\}B=\{x \in L \mid v_L(x) >iv_0\}$. Hence,
$A_iB=\frac{1}{\ap^i}\m L= \js^i \m L$.
\item By (a), for $1 \leq i \leq p-1,
A_i\ap^iB=\frac{1}{\ap^i}\ap^i \m L=\m L$.
Hence, 

$\dis \is = [\cup_{1 \leq i \leq p-1} A_i \ap^i B] \js= \js \m
L$.
\item Let $x \in \m L, v_L(x) >0$. Since the value group is dense in $\R$, there exists an element $y$ of $\m L$ satisfying $0<v_L(y)<v_L(x)/n$. Therefore, $(x) \subset (y^n) \subset \m L^n$ and we can conclude that $\m L = \m L ^n$. The rest follows from (b).
\end{enumerate}
\end{proof}
\begin{rem} In the general case when $e_{L|K}=p, 1 \leq i \leq
p-1$, we have $A_iB \subset \js^i\m L$ and $\is \subset \js \m
L$.
\end{rem}
\begin{pr}\label{ndvr} Under the assumptions above (\Cref{notn}), 
\begin{enumerate}[(a)]
\item $\dis \D^{-1}_{B|A} = \js^{-(p-1)}$
\item $\dis \OL \cong \ol \otimes_B \m L \cong
\frac{\is  }{\is^p}$
\end{enumerate}
\end{pr}
\begin{proof} \begin{enumerate}[(a)]
\item We recall that $\D_{p-1}= \js^{-(p-1)}$ and hence,
$\js^{-(p-1)} \subset \D^{-1}_{B|A}$. If $0 \leq i \leq p-2, \\
y\ap^i \in \D^{-1}_{B|A}\\ \Lr v_L(y)+v_L(x_{p-1-i}) \geq 0 \ $
for all $ x_{p-1-i} \in A_{p-1-i}\\ \Lr v_L(y)+ (p-1-i)v_0 \geq
0$ (since $\Gamma_K$ is dense)

\noindent $\Lr v_L(y\ap^i) \geq -(p-1)v_0\\ \Lr y\ap^i \in
\js^{-(p-1)}$.

\noindent Hence, $\dis \D^{-1}_{B|A} \subset \js^{-(p-1)}$ and
we have the equality $ \D^{-1}_{B|A} = \js^{-(p-1)}$.
\item We defined a map $\pi: \OL \to {\OL}'$ in \Cref{ol'ol}. Let $\Om:=\OL, \Om' := {\OL}'$, for convenience.
Consider the following maps:

\begin{equation}
\xi : \Om' \to \ol \otimes_B \m L \ ; \ \xi(db)=\dl b \otimes
b \end{equation} where $0 \neq b \in \m L$ and \begin{equation}
\psi : \ol \otimes_B \m L \to \Om ; \ \psi(\dl b \otimes
c)=\frac{c}{ab}d(ab)\end{equation} where $b \in L^{\times}, c
\in \m L , a \in K^{\times}; 0 \leq v_L(ab) \leq v_L(c)$. Such
an $a$ exists since $\Gamma_K$ is dense in $\R$.

We verify that these maps are well-defined. Furthermore, $ \xi
\circ \pi \circ \psi: \ol \otimes_B \m L \to \ol \otimes_B \m L$
and $\psi \circ \xi \circ \pi : \Om \to \Om$ are isomorphisms.

\begin{itemize}
\item Let $0 \neq b,c \in \m L , 0 < v_L(c) \leq v_L(b)$. We can
write $b=ch \ ; h \in B$.
\begin{align*}
\dl (b+c) \otimes (b+c)&= \dl c(1+h) \otimes
c(1+h)\\&=(1+h) \dl c \otimes c +(1+h)\dl(1+h) \otimes c
\\&= \dl c \otimes c +h \dl c \otimes c +h \dl h
\otimes c \\&= \dl c \otimes c +h \dl ch \otimes c \\&= \dl c \otimes c + \dl ch \otimes ch \\&= \dl c \otimes c
+\dl b \otimes b
\end{align*}
\item Let $0 \neq b,c \in \m L$

\begin{align*}
\dl (bc) \otimes (bc)&= \dl b \otimes bc + \dl c
\otimes bc \\&= c \dl b \otimes b + b \dl c \otimes c
\end{align*}
\end{itemize}
Thus, $\xi$ is well-defined. Next, we check that $\psi$ is
well-defined.

\begin{itemize}
\item Let $b \in L^{\times} , c \in \m L , a, a' \in K^{\times}$
such that $0 \leq v_L(ab), v_L(a'b) \leq v_L(c)$. Since
$da=0=da'$,

$\frac{c}{ab}d(ab)=\frac{c}{ab}(a db + b da )=\frac{c}{b} db=
\frac{c}{a'b}d(a'b)$. Thus, $\psi$ is independent of choice of
$a$.
\item Let $0 \neq b \in B, c \in \m L , a \in K^{\times}$ as
described in the definition of $\psi$. Since $da=0,$ we have

$\psi (db \otimes c)= b \frac{c}{ab} d (ab)= \frac{c}{a} (a db
+b da)= c db$.

Hence, $\psi$ preserves additivity and Leibniz rule.
\item Let $b, b' \in L, c, c' \in \m L, a, a' \in K^{\times}$
such that $0 \leq v_L(ab) \leq v_L(c)$ and $0 \leq v_L(a'b') \leq
v_L(c')$.

Furthermore, since $\Gamma_K$ is dense in $\R$, we can choose
$a, a'$ such that $0 \leq v_L(aa'bb') \leq v_L(c)$.

$\frac{c}{aa'bb'} d(aa'bb')=\frac{c}{aa'bb'}
[a'b'd(ab)+abd(a'b')]=\frac{c}{ab} d(ab)+\frac{c}{a'b'} d(a'b')$
\end{itemize}
Thus, $\psi$ is well-defined.

Next, we consider the maps $ \xi \circ \pi \circ \psi: \ol
\otimes_B \m L \to \ol \otimes_B \m L$ and $\psi \circ \xi \circ
\pi : \Om \to \Om$.

\begin{itemize}
\item Let $b \in L^{\times}, c \in \m L, a \in K^{\times}, x \in
A$ such that $ 0 \leq v_L(ab) \leq v_L(c)$ and $ab-x \in \m L$.

\begin{align*}
\xi \circ \pi \circ \psi (\dl b \otimes c) &=\frac{c}{ab}\dl (ab-x) \otimes (ab-x)
\\&=\frac{ab-x}{ab}\dl (ab-x) \otimes c
\\&=\frac{ab}{ab}\dl (ab) \otimes c 
\\&=\dl a \otimes c + \dl b \otimes c = \dl b \otimes c
\end{align*}
\item Let $0 \neq b \in B, x \in A$ such that $b-x \in \m L$. 
\begin{align*}
\psi \circ \xi \circ \pi(db) &= \psi( \dl (b-x) \otimes
(b-x)) \\&= \lb \frac{b-x}{b-x} \rb d(b-x) \\&=d(b-x)=db
\end{align*}
\end{itemize}

\noindent This proves the first isomorphism.

\noindent Next, we prove that 

\begin{equation}\label{annh} \ol \cong B/\js^{p-1}
\end{equation}

By \Cref{dlogalpha} (ii), $\ol$ is generated by $\dl \ap = - \dl \lb \frac{1}{\ap}
\rb$. In $\ol$, we have
\begin{align*} 0=-\lb 1-\frac{1}{\ap^{p-1}}\rb \dl
(\frac{1}{f})&=\lb 1-\frac{1}{\ap^{p-1}}\rb \dl f \\&= \lb
1-\frac{1}{\ap^{p-1}}\rb \dl (\ap^p)+ \lb
1-\frac{1}{\ap^{p-1}}\rb \dl \lb 1-\frac{1}{\ap^{p-1}}\rb
\\&= d \lb 1-\frac{1}{\ap^{p-1}}\rb =d \lb
-\frac{1}{\ap^{p-1}}\rb \\&=-d \lb \frac{1}{\ap^{p-1}}\rb= \lb
1-\frac{1}{\ap^{p-1}}\rb \\&=-(p-1)\lb \frac{1}{\ap^{p-1}} \rb \dl \lb \frac{1}{\ap} \rb
\end{align*}
\noindent 
Therefore, $\js^{p-1}=\lb \frac{1}{\ap^{p-1}} \rb$ annihilates
$\ol$.

\noindent Conversely, let $ 0 \neq b \in B$ such that $b \ \ol
=0$. Hence, $ \dis $ for all $ \ 1 \leq i \leq p-1, x_i \in A_i,
b d (x_i \ap^i) =0 \\ \Ra b \in \cap_{i, x_i} G'_{i,x_i}(x_i
\ap^i)B$, where $G_{i,x_i}$ is the minimal polynomial of $x_i
\ap^i$ over $K$.
Let $G:=G_{i,x_i}$ for fixed $(i,x_i)$. Then

\begin{align*}
G'(x_i\ap^i) &= \prod_{1 \leq j \leq p-1} x_i\ap^i\lb 1- \lb
\frac{\ap +j}{\ap} \rb^i \rb \\&= (x_i\ap^i)^{p-1} \prod_{1 \leq
j \leq p-1} \lb 1- \lb \frac{\ap +j}{\ap} \rb^i \rb \\&=
(x_i\ap^i)^{p-1} \prod_{1 \leq j \leq p-1} \lb 1- \lb \frac{\ap
+j}{\ap} \rb \rb u ; \ u \in B^{\times} \\&= (x_i\ap^i)^{p-1}
\lb \frac{-1}{\ap} \rb^{p-1} \ (p-1)! \ u
\end{align*}

\noindent Thus, $\dis b \in \cap_{i, x_i} G'_{i,x_i}(x_i \ap^i)B
= \cap_{i, x_i} (x_i \ap^i)^{p-1}\js^{p-1} \Ra b \in \js^{p-1}$

\noindent By \Cref{annh} and \Cref{ideals},

\noindent $\dis \ol \otimes \m L \cong B/\js^{p-1} \otimes \m L
\cong \m L / \js^{p-1} \m L \cong \js \m L / \js^{p} \m L = \is / \is^p $.

\end{enumerate}
\end{proof}

\section{Appendix: A Non-trivial Example of Defect Extension}
For some of the well-known examples of defect extensions, our main results are trivially true, since the differential modules are all $0$. We construct an  example below that exhibits complications created by the defect.
\begin{ex} Let $k$ be a perfect field of characteristic $p>0$ and let $A_0$ be the local ring of a smooth algebraic surface over $k$ at some closed point, and assume that we are given an Artin-Schreier extension $L$ of the field of fractions $K$ of $A_0$ given by 
$$\alpha^p- \alpha= \frac{a+y}{x^n}$$
where $x$ and $y$ are regular parameters of $A_0$, $a\in k\setminus \F_p$, and  $n\geq 1$ is coprime to $p$. 
We assume $n\geq 3$ if $p=2$. We will construct two dimensional regular local rings $A_i\subset K$ ($i\geq 0$) such that
$$A_0\subset A_1\subset A_2  \subset \dots$$
as follows, by using successive blow ups. We will have a valuation ring $A:=\bigcup_i A_i$ for which this Artin-Schreier extension has defect.
\end{ex}

\subsection{Construction}
Let $u:=y+a$. 
Define $x'\in K$ by $x=x'y^p$ and let $A'_0$ be the 
 the local ring of $A_0[x']\subset  K$ at the maximal ideal generated by $x'-1$ and $y$. 
  Then $A'_0$  is a two dimensional regular local ring with regular parameters $x'-1$ and $y$.  Since $n$ is coprime to $p$, $z:= (x')^{-n}-1$ and $y$ are also regular parameters of $A'_0$. Define $z'\in K$ by $z=z'y$ and let $A_1$ be the local ring of $A'[z']\subset K$ at the maximal ideal generated by $z'-1$ and $y$. 

Then the above Artin-Schreier equation is rewritten as follows. 
We have 
\begin{equation} \begin{split} f_0 & :=\frac{a+y}{x^n} =\frac{a+y}{(x')^ny^{np}}= \frac{(a+y)(1+z'y)}{y^{np}}\\ & =\frac{a}{y^{np}} + \frac{a+1+ a(z'-1)+z'y}{y^{np-1}}\\ &
=  \frac{a+1+y_1}{x_1^{np-1}} + c^p-c \end{split} \end{equation}

with
$$x_1=y,\quad  y_1= a(z'-1)+ z'y+a^{1/p}y^{n(p-1)-1}, 
 \quad c= a^{1/p}y^{-n}.$$
In $A_1$, $x_1$ and $y_1$ are regular parameters, and 
 the same Artin-Schreier extension is obtained by 
$$\alpha_1^p-\alpha_1= \frac{a+1+y_1}{x_1^{np-1}}=:f_1.$$

We can repeat this process and get $A_0\subset A_1\subset A_2\subset \dots$ inductively. To sum up, we have the following for all $i \geq 0$:

 In $A_i$, the regular parameters are $x_i$ and $y_i$, as described in the construction (and we put $x=x_0, y=y_0, \ap = \ap_0, n=n_0$). The same Artin-Schreier extension is give by $$\alpha_i^p-\alpha_i= \frac{a+i+y_i}{x_i^{n_i}}=:f_i$$ where the integers $n_i$ satisfy the recursive relation $n_{i+1}=pn_i-1$.

\subsection{Valuation on $A$ and $B$: } Let $B$ be the integral closure of $A$ in $L$.
Due to their construction using successive blow ups we note that
$A$ and $B$ are valuation rings \cite{A}. Let $v_K=v$ be the valuation on
$K$. We see from the calculations below that the value group of
$K$ is $\dis \Gamma \cong \Z[\frac{1}{p}]; v(x_0) \mapsto
1$.
For all $i \geq 0$, we have the following:
\begin{enumerate}
\item $\dis n_i=p^in-(p^{i-1}+ \dots +p^2+p+1)=p^in-\frac{p^i-1}{p-1}$.
\item $v(x_i)=pv(y_i)=pv(x_{i+1})$\\ And hence, we get $$v(x_i)= \frac{1}{p^i} , v(y_i)= \frac{1}{p^{i+1}}$$
\item $\dis -v(f_i)=n_i v(x_i)=n-\frac{1}{p-1}+\frac{1}{p^i(p-1)}$.
\end{enumerate}

Since $\Gamma$ is $p$-divisible, $L|K$ has defect. We will use
$v$ to denote $v_L$ as well. By the computations above, it follows that $$-v(\ap_i)= \frac{1}{p} \lb -v(f_i) \rb = \frac{n}{p}-\frac{1}{p(p-1)}+\frac{1}{p^{i+1}(p-1)}$$

\subsection{Special Ideals and Differential Modules} 

Due to the defect, we have $\is=\js$ and it is enough to look
at $\Om^1$'s instead of $\om^1$'s (see \Cref{ij}).

The elements $\dis \frac{1}{\ap_i}$ for $i \geq 0$ generate the
ideal $\js$ of $B$ and the elements $\dis \frac{1}{f_i}$ for
$i \geq 0$ generate the ideal $H$ of $A$. 
\begin{itemize}
\item Since $\dis \inf_{i\geq 0} \lb\frac{n}{p}-\frac{1}{p(p-1)}+\frac{1}{p^{i+1}(p-1)}\rb= \frac{n}{p}-\frac{1}{p(p-1)}$, we have\\ $\dis \is=\js=\{b \in B \mid
v(b)>\frac{1}{p}(n-\frac{1}{p-1} )=:v_0\}$, and
consequently,
\item $\dis \ns= \{a \in A \mid v(a)>
(n-\frac{1}{p-1})=pv_0\}$.

\item Since $\dis \inf_{i\geq 0} -v(f_i)= \inf_{i\geq 0} \lb n-\frac{1}{p-1}+\frac{1}{p^i(p-1)}\rb = n-\frac{1}{p-1}$, there is no best $f$ and furthermore,\\ $ \dis H= \{a \in A \mid  v(a)>
(n-\frac{1}{p-1})\}$
\end{itemize}

\noindent Thus, \Cref{hn} is clearly true in this case.\\\\ Next, use the notation from the proof of \Cref{comm dia} and
consider the differential modules $\Om^1_{B|A}, \Om^1_{B_i|A_i}$'s.

\noindent  Let $\bt_i:=\ap_iy_i^{n_i}$. Then the integral closure of $A_i$ in $L$ is given by $B_i=A_i [\bt_i]$. Let  $F_i(T)$ be  the minimal polynomial of $\bt_i$
 over $A_i$. Then $F_i'(T)=-y_i^{n_i(p-1)}$.\\
We have an isomorphism :
$A_i[\bt_i]/F_i'(\bt_i) \to \Om^1_{B_i|A_i}$ of $B_i$-modules  via the
$A_i$-linear map $a \mapsto a d\bt_i$; for all $ a \in A_i$.\\\\
We use $\ap_0$ as our starting point. Valuation of $\ap_0$ is $-\mu= -n/p$ and $\dis v_0=\frac{1}{p}\lb n-\frac{1}{p-1}\rb$.
The fractional ideals $\T$ and $\T'$ of $L$ are described by
$\dis \T= \{x \in L \mid v(x) > -\frac{1}{p(p-1)}=:v_1 \}$
and $\dis \T'=\{ x \in L \mid v(x) > \lb\frac{n(p-1)-1}{p}\rb
+v_1=:v_2\}$. Then we have:
\begin{itemize}
\item $\Om^1_{B|A} \cong \T/\T'$ 
\item $\T/ \js \T \cong \js/\js^2$
\end{itemize}
From this, \Cref{comm dia} will follow.
\\\\We can also verify that 
\begin{itemize}
\item $\D^{-1}_{B|A}= \cap_{i \geq 0} \D^{-1}_{B_i|A_i}$.
\item $\D^{-1}_{B|A}= \{x \in L \mid v(x) > -(p-1)v_0 \}$ 
\item $\D_{B|A}= \js^{p-1}$ is the annihilator of $\Om^1_{B|A}$.\end{itemize}

\newpage
\textbf{Acknowledgments:} I am extremely grateful to my advisor
Professor Kazuya Kato for his invaluable advice, helpful
feedback during the writing process, and his constant support
during the project. I would also like to thank the anonymous referee  for their detailed and constructive comments on an earlier draft.

\medskip
\medskip
\medskip
\noindent Vaidehee Thatte\\Department of Mathematics,\\
University of Chicago,\\ 5734 University Ave.\\ Chicago,
Illinois 60637\\
Email: \texttt{vaidehee@math.uchicago.edu}

\vfill

\end{document}